\documentclass[a4paper,11pt]{article}
\usepackage[T1]{fontenc}
\usepackage[utf8]{inputenc}
\usepackage[italian,english]{babel}
\usepackage{microtype}
\usepackage{quoting}
\usepackage{amsthm, amsmath, amssymb, amscd, mathrsfs, mathtools}

\usepackage[a4paper]{geometry}
\usepackage{booktabs, caption, graphicx} 
\usepackage{enumitem}
\usepackage[framemethod=default]{mdframed}
\usepackage{braket}
\usepackage{tikz-cd}
\usepackage{centernot}
\usepackage{fancyhdr}
\usepackage{emptypage}
\usepackage{esint}
\usepackage[english]{varioref}

\usepackage[backend=biber,style=alphabetic]{biblatex}
\usepackage{csquotes}
\global\mdfdefinestyle{exampledefault}{}

\usepackage{settings} 

\title{On a conjecture of De Giorgi about the phase-field approximation of the Willmore functional}
\author{
Giovanni Bellettini\footnote{
Dipartimento di Ingegneria dell'Informazione e Scienze Matematiche,
Universit\`a di Siena, 53100 Siena, Italy,
and International Centre for Theoretical Physics ICTP,
Mathematics Section, 34151 Trieste, Italy.
E-mail: bellettini@diism.unisi.it
                      }\and
Mattia Freguglia\footnote{Scuola Normale Superiore, 56126 Pisa, Italy. E-mail: mattia.freguglia@sns.it} \and
Nicola Picenni\footnote{Scuola Normale Superiore, 56126 Pisa, Italy. E-mail: nicola.picenni@sns.it}}
\date{}

\begin{document}

\maketitle

\begin{abstract}

In 1991 De Giorgi conjectured that, given $\lambda >0$, if $\mu_\eps$ stands for the density of the Allen-Cahn energy and $v_\eps$ represents its first variation, then $\int [v_\eps^2 + \lambda] d\mu_\eps$ should $\Gamma$-converge to $c\lambda \Per(E) + k \mathcal{W}(\Sigma)$ for some real constant $k$, where $\Per(E)$ is the perimeter of the set $E$, $\Sigma=\partial E$, $\mathcal{W}(\Sigma)$ is the Willmore functional, and $c$ is an explicit positive constant. A modified version of this conjecture  was proved in space dimensions $2$ and $3$ by Röger and Schätzle, when the term $\int v_\eps^2 \, d\mu_\eps$ is replaced by $  \int  v_\eps^2 {\eps}^{-1} dx$, with a suitable $k>0$. In the present paper we show that, surprisingly, the original De Giorgi conjecture holds with $k=0$. Further properties on the limit measures obtained under a uniform control of the approximating energies are also provided.

\vspace{4ex}

\noindent{\bf Mathematics Subject Classification 2020 (MSC2020):} 
49J45, 49Q20.

\vspace{4ex}


\end{abstract}


\section{Introduction}

In \cite[Conjecture 4]{DeGiorgi} De Giorgi posed the following:

\begin{conjecture}
\label{conj_DG}
Let $n\geq 2$ be an integer number and let $E\subset \R^n$ be a set whose boundary $\Sigma:=\partial E$ is a hypersurface of class $C^2$. For any open set $\Omega\subset \R^n$ and any positive number $\lambda>0$ let us consider the following family of functionals, indexed by the parameter $\eps>0$,
\begin{equation}\label{def:DG_eps_cos}
\mathcal{DG}_\eps(u,\Omega):=\int_{\Omega} \left[\left(2\eps\Delta u-\frac{\sin u}{\eps} \right)^2+\lambda\right] \left[\eps|\nabla u|^2 + \frac{1-\cos u}{\eps} \right]\,dx,
\end{equation}
if $u \in W^{2,1}(\Omega)$, and $\mathcal{DG}_\eps(u):=+\infty$, if $u \in L^1(\Omega)\setminus W^{2,1}(\Omega)$.

Then there exists a constant $k \in \R$ such that
$$\Gamma(L^1(\Omega))-\lim_{\eps \to 0^+} \mathcal{DG}_\eps(2\pi \chi_E,\Omega)= c \lambda \mathcal{H}^{n-1}(\Sigma\cap \Omega) + k\int_{\Sigma\cap\Omega} H^2 d\mathcal{H}^{n-1},$$
where $\chi_E$ is the characteristic function of the set $E$ (that is equal to one inside $E$ and null outside), $c=8\sqrt{2}$, $H(y)$ is the mean curvature of $\Sigma$ at the point $y$ and $\mathcal{H}^{n-1}$ stands for the $(n-1)$-dimensional Hausdorff measure in $\mathbb{R}^n$.
\end{conjecture}

We point out that we have added a factor $2$ in front of the laplacian so that, if $u \in W^{2,2}(\Omega)$, the squared term is really the $L^2$-gradient of the Allen-Cahn energy, that is the functional
\begin{equation}\label{def:Eeps}
E_\eps (u,\Omega):=\int_{\Omega} \left[\eps|\nabla u|^2 + \frac{W(u)}{\eps} \right]\,dx,    
\end{equation}
if $u \in W^{1,2}(\Omega)$, and $E_\eps(u) := +\infty$, if $u \in L^1(\Omega)\setminus W^{1,2}(\Omega)$. Here $W \colon \R \to  [0,+\infty)$ is a multiple-well potential, like $W(u)=1-\cos u$, as in the conjecture of De Giorgi, or the more popular double-well potential $W(u)=(1-u^2)^2$. Precise assumptions on $W$ will be stated in Section~\ref{section:assump-state}.

The $\Gamma$-convergence of the family $E_\eps$ is the object of the celebrated Modica-Mortola theorem (see \cite{DG-Fra}, \cite{Modica-Mortola} and \cite{Modica}) which in this case says that
$$\Gamma(L^1(\Omega))-\lim_{\eps \to 0^+} E_\eps\left(\chi^{a,b} _E,\Omega\right)= \sigma_W ^{a,b} \Per(E,\Omega),$$
where now $E$ has finite perimeter $\Per(E,\Omega)$ in $\Omega$, $a<b$ are two consecutive zeros of $W$, $\chi^{a,b} _E$ is a suitable modification of the characteristic function, defined as
\begin{equation}\label{def:chi}
\chi^{a,b} _E (x):=\begin{cases}
a &\mbox{if }x \notin E,\\
b &\mbox{if }x \in E,
\end{cases}    
\end{equation}
and
\begin{equation}\label{def:sigma}
\sigma ^{a,b} _W :=2\int_a ^b \sqrt{W(u)}\, du.    
\end{equation}

We observe that in the case $W(u)=1-\cos u$ it turns out that $\sigma^{0,2\pi}_W=c$, so De~Giorgi's conjecture is actually saying that the functional
\begin{equation}\label{def:Geps}
G_\eps (u,\Omega):= \int_{\Omega} \left[2\eps\Delta u-\frac{W'(u)}{\eps} \right]^2 \left[\eps|\nabla u|^2 + \frac{W(u)}{\eps} \right]\,dx    
\end{equation}
is an approximation for a multiple of the Willmore functional
\[
    \mathcal{W}(\Sigma,\Omega):=\int_{\Sigma \cap \Omega} H^2 d\mathcal{H}^{n-1},
\]
provided $\Sigma$ is of class $C^2$.

This seems reasonable because the mean curvature is known to represent the first variation of the perimeter and the term $2\eps \Delta u - W'(u)/\eps$ represents the gradient of the functional $E_\eps$. 
Moreover, if $\{u_\eps\}$ is a family of functions that converges to $\chi_E ^{a,b}$ in $L^1$, then the energy densities, that are the (normalized) measures
\begin{equation}\label{def:mueps}
\mu_\eps:=\frac{1}{\sigma^{a,b}_W}\left[\eps|\nabla u_\eps|^2 + \frac{W(u_\eps)}{\eps} \right]\Leb^n,
\end{equation}
where $\Leb^n$ is the $n$-dimensional Lebesgue measure, in the limit should be larger than or equal to the measure $\mathcal{H}^{n-1} \mres \Sigma$, as a consequence of the $\Gamma$-convergence of $E_\eps$.

\paragraph{}
In the case $W(u)=(1-u^2)^2$, it is proved in \cite{Bel-Pao-Bologna} that this is what actually happens when one considers the usual recovery sequences for the $\Gamma$-limit of $E_\eps$, that consist of a slight modification of the functions $u_\eps(x):=q_0(\dsig(x)/\eps)$, where $q_0$ is the optimal one-dimensional profile and $\dsig$ is the signed distance from $\Sigma=\partial E$ (precise definitions are given in the next section). More specifically, an estimate from above for the $\Gamma$-$\limsup$ of $\mathcal{DG}_\eps$ with a positive constant $k>0$ was proved.

Moreover, the authors of \cite{Bel-Pao-Bologna} proposed to investigate the functional
\[
    \widehat{G}_\eps(u,\Omega):=\int_{\Omega} \frac{1}{\eps}\left(2\eps\Delta u-\frac{W'(u)}{\eps} \right)^2 \,dx,
\]
in place of $G_\eps$, in order to simplify the problem, and they proved a $\Gamma$-$\limsup$ estimate (with a positive constant $k>0$) also for the functionals $E_\eps+\widehat{G}_\eps$.

The modification is motivated by the fact that the second factor in the integrand of $G_\eps$
should be proportional to $\eps^{-1}$ near $\Sigma$, while the contribution of both factors far from this boundary should not be relevant for the $\Gamma$-limit.

However, the $\Gamma$-$\liminf$ estimate turned out to be much more involved  and, after some partial results by Bellettini and Mugnai \cite{Bellettini-Mugnai} and Moser \cite{Moser}, the problem has been solved in dimensions 2 and 3 by Röger and Schätzle \cite{Roger-Schatzle}, and reproved differently in dimension 2 by Nagase and Tonegawa \cite{Nagase-Tonegawa}, while it is still open in higher dimensions. More precisely, in the special case $W(u)=(1-u^2)^2$ and $n \in\{2,3\}$, Röger and Schätzle were able to prove that if $\Sigma$ is of class $C^2$ then
\[
    \Gamma(L^1(\Omega))-\lim_{\eps \to 0^+} (E_\eps+\widehat{G}_\eps)\left(\chi^{-1,1}_E,\Omega\right)= \sigma^{-1,1}_W \mathcal{H}^{n-1}(\Sigma\cap \Omega) + k\int_{\Sigma\cap\Omega} H^2\, d\mathcal{H}^{n-1},
\]
for some positive constant $k>0$. Moreover, they also proved (see Theorem~4.1 and Theorem~5.1 in \cite{Roger-Schatzle}) that if $\{u_\eps\} \subset W^{2,2}(\Omega)$ is a family of functions for which
$$E_\eps(u_\eps, \Omega)+\widehat{G}_\eps(u_\eps,\Omega)\leq C,$$
then the weak* limit (up to subsequences) of the measures $\mu_\eps$ is an integral $(n-1)$-varifold with some estimates on the curvature depending on the weak* limit of the measures 
$$\alpha_\eps:=\frac{1}{\eps}\left(2\eps\Delta u_\eps-\frac{W'(u_\eps)}{\eps} \right)^2 \Leb^n.$$

As the authors explain in the introduction, a crucial step in their proof is the control of the so-called discrepancy measures
\begin{equation}\label{def:discrepancy}
\xi_\eps:=\left[ \eps |\nabla u_\eps|^2 - \frac{1}{\eps}W(u_\eps) \right] \Leb^n,
\end{equation}
that is obtained in Proposition 4.4 and Proposition 4.9 in \cite{Roger-Schatzle}.

For an exhaustive list of references about the approximation of the Willmore functional and other variants of this model we refer to \cite{BMO} and to the recent paper \cite{Ratz-Roger}, where the interested reader can also find many numerical simulations.

\paragraph{}
The main result of this paper is a proof that, surprisingly, De Giorgi's conjecture holds true with $k=0$. This means that, as opposite to $\widehat{G}_\eps$, the functional $G_\eps$ does not contribute to the $\Gamma$-limit of $\mathcal{DG}_\eps$ that, instead, turns out to be the same as the one obtained with the functionals $\lambda E_\eps$ alone, and this holds with a quite general class of potentials $W$. This also implies that Conjecture 5 in \cite{DeGiorgi} does not hold, because the perimeter alone, if considered as a function of $\Omega$, is clearly subadditive.

The proof of course consists in finding a family $\{u_\eps\}\subset W^{2,1}_{\mathsf{loc}}(\R^n)$ of functions converging in $L^1$ to $\chi^{a,b}_E$ for which
\[
    \lim_{\eps\to 0^+} E_\eps(u_\eps,\R^n) = \sigma^{a,b}_W \cdot \mathcal{H}^{n-1}(\Sigma)
    \quad \text{and} \quad
    \lim_{\eps\to 0^+} G_\eps(u_\eps,\R^n)=0.
\]

We construct these functions by perturbing the classical recovery sequences for $E_\eps$. In particular, we need to modify the optimal one-dimensional profile $q_0$ in such a way that the two factors in the functional $G_\eps$ concentrate in different regions, so that their product becomes small. We do this by means of a suitable differential equation that prescribes the discrepancy measures in dimension one, providing us with the required perturbed one-dimensional profile, that can be further modified with cut-off functions, as in the classical theory by Modica and Mortola, to produce the final family $\{u_\eps\}$.

This would be enough in the case when $\Sigma$ is a sphere, hence it has constant mean curvature, but it turns out that the general case is more delicate. Indeed, in this case the perturbation of the optimal profile has to be adjusted depending on the local geometry of $\Sigma$. We do this by adding a parameter in the equation for the perturbed profile, in order to gain more flexibility in the construction of the recovery sequence.

We recall that in the functional $\widehat{G}_\eps$ the second factor has been replaced by the constant $\eps^{-1}$ so our strategy, that allows the first factor to be very large in regions where the other one is small, is not effective in decreasing the value of the modified functional, because in this case such regions do not exist (actually $\widehat{G}_\eps(u_\eps, \mathbb{R}^n) \to +\infty$ for our choice of $\{u_\eps\}$).

As a corollary of our main result, we obtain that the limit of the energy densities $\mu_\eps$ is not necessarily $(n-1)$-rectifiable, even if the functionals are equibounded. In fact, it can also happen that these measures converge to a Dirac mass or, more generally, to a measure that is not absolutely continuous with respect to $\mathcal{H}^{n-1}$.

In the opposite direction, despite this unexpected result, it seems that the boundedness of the family $\{G_\eps(u_\eps,\R^n)\}$ still carries some information on the behavior of the energy densities. More specifically, we believe that it could prevent the diffusion of the energy on large sets, while in general the limit of $\mu_\eps$ under the only assumption of the boundedness of the energy $E_\eps(u_\eps,\R^n)$ can be any positive finite measure. Indeed, in the toy model of radial symmetry, if we remove the origin (where a Dirac mass could appear), we can prove that if $\{u_\eps\}\subset W^{2,1}_{\mathsf{loc}}(\R^n)\cap W^{1,2}_{\mathsf{loc}}(\R^n)$ is a family of functions with
\[
    E_\eps(u_\eps,\R^n)+G_\eps(u_\eps,\R^n)\leq C,
\]
then any weak* limit of $\mu_\eps$ is an integral $(n-1)$-varifold if restricted to $\R^n\setminus \{0\}$ (which of course in this case is simply a union of concentric spheres). The proof of this fact is based on a blow-up argument, similar to the one in \cite{Bellettini-Mugnai}.

We observe that the radial symmetry and the removal of the origin automatically imply that the limit measure is absolutely continuous with respect to $\mathcal{H}^{n-1}$, but these assumptions do not prevent a priori that this measure may be supported on sets with larger dimension.
In particular, if one only assumes the boundedness of the energies $E_\eps(u_\eps,\R^n)$, without additional assumptions on $G_\eps(u_\eps,\R^n)$, then the limit of the energy densities can be any positive finite radially symmetric measure, so the integrality of the limit measure is not trivial.

We point out that the radial symmetry is not even ruling out the ``pathology'' that leads to the disappearance of $G_\eps$ in the limit, in fact the recovery sequence for the $\Gamma$-limit of $\mathcal{DG}_\eps$ when $E$ is a ball can be made of radially symmetric functions.

For this reason, we think that this positive result could be true even in the general case, in the sense that we expect that, if $G_\eps(u_\eps,\R^n)$ is uniformly bounded, then the limit of the energy densities should be concentrated on a $\mathcal{H}^{n-1}$-finite set. However, a proof of this fact without the radial assumption would probably be much more complicated, as it happens in the case of the modified functionals considered in \cite{Roger-Schatzle}.

\paragraph{}
The paper is organized as follows. In Section 2 we introduce some notation, we recall some useful facts from the literature and we prove some preliminary results. In particular, in Subsection~\ref{subse:perturbed-one-dim} we prove all the properties of the one-parameter family of ODEs needed to construct the sequence $\{ u_{\eps} \}$ which makes $G_{\eps}(u_{\eps}, \mathbb{R}^n)$ infinitesimal.
In Section 3 we prove our main result (Theorem~\ref{teo:gammalim}), that is the computation of the $\Gamma$-limit of the functionals $\mathcal{DG}_\eps$.
In Section 4 we prove our integrality result in the radially symmetric case (Theorem~\ref{theorem:main2}).

\section{Statements and preliminary lemmas}
\label{section:assump-state}

In this section we introduce the precise setting of our work and we prove some preliminary lemmas.

\subsection{Assumptions on $W$ and main results}

First of all, we state our assumptions on the potential, that is a function $W:\R \to [0,+\infty)$ with the following properties:
\begin{description}
    \item[(W1)] $W \in C^2(\R)$,
    \item[(W2)] there exists an interval $[a,b]\subset \R$ such that $W(a)=W(b)=0$ and $W(u)>0$ for every $u \in (a,b)$,
    \item[(W3)] $W''(a)>0$ and $W''(b)>0$.
\end{description}

We state also an extra assumption, stronger than (W3), that we need only for proving the results of Section~\ref{section:blow-up}.

\begin{description}
    \item[(W3+)] There exists $\kappa > 0$ such that $W''(u) \ge 2\kappa^2$ for every $u \in \mathbb{R} \setminus (a,b)$.
\end{description}

We observe that the classical double-well potential $W(u) = (1-u^2)^2$ satisfies assumption (W3+) with $\kappa = 2$. On the other hand it is evident that the potential $W(u) = 1 - \cos u$ in De~Giorgi's conjecture does not satisfy this condition.

\begin{remark}
    \label{remark:W4}
    \rm Combining assumption (W3+) together with $W(a)=W(b)=0$ we deduce that $W(u) \ge \kappa^2 (u-b)^2$ for every $u \ge b$ and $W(u) \ge \kappa^2 (u-a)^2$ for every $u \le a$, in particular $W$ has no zeros outside $[a,b]$. Moreover, in this case it holds 
    \begin{equation}
        \label{eq:W4-revisited}
        \frac{\abs*{W'(u)}}{\sqrt{W(u)}} \ge \kappa \quad \text{for every $u \in \mathbb{R} \setminus [a,b]$}.
    \end{equation}
    To prove~(\ref{eq:W4-revisited}) we notice that assumption (W3+) implies that $W'$ is increasing in $[b,+\infty)$ (and also in $(-\infty,a]$), hence
    \[
        W(u) = W(b) + \int_b^u W'(t) dt \le W'(u) (u-b) \quad \text{for every $u \in [b,+\infty)$}.
    \]
    Therefore, if $u > b$ we easily deduce from the previous inequality that
    \[
        \frac{W'(u)}{\sqrt{W(u)}} \ge \frac{W(u)}{u-b} \cdot \frac{1}{\sqrt{W(u)}} = \frac{\sqrt{W(u)}}{u-b} \ge \kappa,
    \]
    where in the last inequality we used $W(u) \ge \kappa^2 (u-b)^2$. The case $u < a$ is similar.
\end{remark}

Now we can state our main results.

\begin{theorem}\label{teo:gammalim}
Let $W \colon \R \to [0,+\infty)$ be a potential satisfying (W1), (W2) and (W3) and let $E\subset \R^n$ be a set with finite perimeter. Then there exists a family $\{u_\eps\}\subset C^2 (\R^n)$ of functions such that
$$\lim_{\eps \to 0^+}\left\|u_\eps - \chi^{a,b}_E\right\|_{L^1(\R^n)}= 0,$$
and
$$\lim_{\eps\to 0^+} \Big(\lambda E_\eps(u_\eps,\R^n)+G_\eps(u_\eps,\R^n) \Big) =\lambda \sigma_W ^{a,b} \Per(E),$$
where $\chi^{a,b}_E$ and $\sigma_W ^{a,b}$ are defined respectively in (\ref{def:chi}) and in (\ref{def:sigma}).
\end{theorem}

Since the $\Gamma$-$\liminf$ estimate is an immediate consequence of the Modica-Mortola theorem, this result implies the validity of De Giorgi's conjecture with $k=0$, actually with more general sets $E$ and more general potentials. Moreover, we have the following corollary.

\begin{corollary}\label{cor:mueps_to_delta}
There exists a family $\{u_\eps\} \subset C^2(\R^n)$ of functions such that
    \[
        \limsup_{\eps\to 0^+} \Big( E_\eps(u_\eps,\R^n)+G_\eps(u_\eps,\R^n)\Big) <+\infty,
    \]
and
    \[
        \mu_\eps \overset{*}{\weakto} \delta_0 \ \text{in duality with $C_0(\R^n)$},
    \]
where $\mu_\eps$ are defined in (\ref{def:mueps}), $\delta_0$ denotes the Dirac mass centered at zero and $C_0(\R^n)$ is the space of continuous functions vanishing at infinity.
\end{corollary}

\begin{proof}
For any positive integer $k \in \N^+$ and for any $j \in \{-k,\dots,k\}$, let us set
$$r_{k,j}:=\left[\frac{1}{\omega_{n-1}}\left(\frac{1}{2k+1} + \frac{j}{2k(2k+1)}\right)\right]^\frac{1}{n-1},$$
where $\omega_{n-1}$ denotes the $\mathcal{H}^{n-1}$ measure of the unit sphere in $\mathbb{R}^n$. Let us consider the hypersurfaces
$$\Sigma_k:=\bigcup_{j=-k} ^{k} \partial B_{r_{k,j}},$$
and the bounded sets $E_k$ such that $\Sigma_k=\partial E_k$, namely
$$E_k=(B_{r_{k,k}}\setminus B_{r_{k,k-1}}) \cup (B_{r_{k,k-2}}\setminus B_{r_{k,k-3}}) \cup \dots \cup (B_{r_{k,-k+2}}\setminus B_{r_{k,-k+1}}) \cup B_{r_{k,-k}},$$
where all balls are centered in the origin.

We observe that for any $k \in \N^+$ it holds
$$\Per (E_k)=\mathcal{H}^{n-1}(\Sigma_k)=\sum_{j=-k} ^k \omega_{n-1} r_{k,j}^{n-1}=1.$$

By Theorem \ref{teo:gammalim}, for any $k\in \N^+$ we can find a family $\{u^k _\eps\}\subset C^2(\R^n)$ of functions such that
\[
    \lim_{\eps \to 0^+}\left\|u_\eps^k - \chi^{a,b}_E\right\|_{L^1(\R^n)}= 0,
    \quad
    \lim_{\eps\to 0^+} G_\eps(u^k _\eps,\R^n)= 0,
\]
and
\[
    \lim_{\eps\to 0^+} \mu_\eps ^k (\R^n)
    =
    \frac{1}{\sigma_W ^{a,b}} \lim_{\eps\to 0^+} E_\eps(u^k _\eps,\R^n)= \Per(E_k) = 1,
\]
where $\mu^k _\eps$ is the energy density associated to $u^k _\eps$.

On the other hand, for every $r>r_{k,k}$ from Modica-Mortola theorem we deduce that
$$\liminf_{\eps \to 0^+} \mu^k _\eps (B_{r})=\frac{1}{\sigma_W ^{a,b}} \liminf_{\eps \to 0^+} E_\eps (u^k _\eps,B_r)\geq \Per(E_k,B_r)=1,$$
and hence
$$\lim_{\eps\to 0^+} \mu_\eps ^k (\R^n\setminus B_r)=\lim_{\eps\to 0^+} \Big(\mu_\eps ^k (\R^n) - \mu^k _\eps (B_{r})\Big)=0.$$

Therefore, with a diagonal procedure we can find a family $\{k_\eps\}$ of positive integer numbers such that $k_\eps \to +\infty$ as $\eps \to 0^+$ for which, if we set $u_{\eps} := u_{\eps}^{k_\eps}$ and $\mu_{\eps} := \mu_{\eps}^{k_{\eps}}$, then
$$\lim_{\eps\to 0^+} G_\eps(u_\eps,\R^n)= 0, \quad \lim_{\eps\to 0^+} E_\eps(u_\eps ,\R^n)= \sigma_W ^{a,b},$$
and such that for every $r>0$
$$\lim_{\eps \to 0^+} \mu_\eps(B_r)=1, \quad \lim_{\eps \to 0^+} \mu_\eps(\R^n \setminus B_r)=0.$$
This clearly implies that $\mu_\eps \overset{*}{\weakto} \delta_0$.
\end{proof}

\begin{remark}\label{rem:limit_mueps}
\rm Using more or less the same argument (with tubular neighborhoods instead of balls) one can obtain as limit of $\mu_\eps$ all the measures of the kind $\mathcal{H}^d \mres K$, where $K\subset \R^n$ is a smooth and closed $d$-dimensional submanifold of $\R^n$, for some $d \in \{0,\dots,n-2\}$.
With some additional effort (for example perturbing the tubular neighborhoods), it should be possible to obtain also more general measures supported on submanifolds with codimension larger than one and probably also more general classes of measures concentrated on sets that are $\mathcal{H}^{n-1}$ negligible.
\end{remark}

Now we state our second main result which suggests that, despite the fact that $G_\eps$ vanishes in the $\Gamma$-limit and the examples of Corollary \ref{cor:mueps_to_delta} and Remark \ref{rem:limit_mueps}, its boundedness still restricts in some way the class of possible limits of the energy densities. Unfortunately, we are able to prove such a result only in the simplified case of radially symmetric functions.

\begin{theorem}
\label{theorem:main2}
Let $W \colon \R \to [0,+\infty)$ be a potential satisfying (W1), (W2) and (W3+). Let $\{u_\eps\} \subset W_{\mathsf{loc}}^{2,1}(\R^n)\cap W_{\mathsf{loc}}^{1,2}(\R^n)$ be a family of radially symmetric functions such that
\[
    \limsup_{\eps\to 0^+} \Big( E_\eps(u_\eps,\R^n)+G_\eps(u_\eps,\R^n) \Big) <+\infty,
\]
and let $\mu_\eps$ and $\xi_\eps$ be defined respectively in (\ref{def:mueps}) and in (\ref{def:discrepancy}).

Then, for any sequence $\eps_k \to 0^+$ there exist an at most countable index set $I$ and a family $\{ r_i \}_{i \in I} \subset (0,+\infty)$ of radii such that
\begin{equation}\label{eq:somma_ri_finita}
\sum_{i \in I} r_i ^{n-1} <+\infty,
\end{equation}
and, up to (not relabelled) subsequences,
\begin{equation}\label{eq:int-var}
\mu_{\eps_k} \overset{*}{\weakto} \sum_{i \in I} \mathcal{H}^{n-1} \mres \partial B_{r_i} \
    \text{in duality with $C_0(\mathbb{R}^n \setminus \{0\})$.}    
\end{equation}
Moreover, we have
\begin{equation}\label{eq:discr_to_0}
\xi_{\eps} \overset{*}{\weakto}0 \
    \text{in duality with $C_0(\mathbb{R}^n \setminus \{0\})$.}
\end{equation}
\end{theorem}

The vanishing of the discrepancy measures in the limit is an important property, usually referred to as ``equipartition of energy'' because it means that the two addenda in the Allen-Cahn energy are asymptotically equal. The proof of this property is a crucial step in many results involving the Allen-Cahn energy as, for example, in \cite{Ilmanen} and \cite{Roger-Schatzle}.

In our context it is not clear whether the uniform boundedness of $E_\eps(u_\eps,\R^n)+G_\eps(u_\eps,\R^n)$ is enough to deduce this property, since we are not able to extend our argument to the non-radial case, and actually even in the radial case we cannot exclude that the discrepancy measures concentrate at the origin.

On the other hand, we do not have examples in which the functionals are uniformly bounded but the discrepancy measures do not vanish in the limit. Indeed, even for the family that we use to prove Theorem~\ref{teo:gammalim} we have equipartition of energy and, as a consequence, the same holds for the family of Corollary~\ref{cor:mueps_to_delta}. Therefore the following natural question remains open, even if we consider only families of smooth functions.

\begin{question}
Let $\{u_\eps\}\subset W^{2,1}_{\mathsf{loc}}(\R^n) \cap W_{\mathsf{loc}}^{1,2}(\R^n)$ be a family of functions such that
\[
    \limsup_{\eps\to 0^+} \Big( E_\eps(u_\eps,\R^n)+G_\eps(u_\eps,\R^n) \Big) <+\infty.
\]
Is it true that $\xi_{\eps} \overset{*}{\weakto}0$?
\end{question}

\subsection{The perturbed one-dimensional profile}
\label{subse:perturbed-one-dim}

Now we recall the definition of the one-dimensional optimal profile, that is the solution $q_0 \colon \R\to\R$ of the following ordinary differential equation:
\begin{equation}\label{def:q0}
\begin{cases}
\dot{q}_0 (s) = \sqrt{W(q_0(s))}& \forall s\in \R, \\
q_0(0)=\frac{a+b}{2}.
\end{cases}\end{equation}

The assumptions on the potential $W$ ensure that $q_0$ satisfies the following properties (see for instance \cite{BNN}):

\begin{lemma}\label{lemma:q0}
Let $W \colon \R\to [0,+\infty)$ be a potential satisfying (W1), (W2) and let $q_0$ be the solution of (\ref{def:q0}). Then $q_0$ is well defined on the whole real line and
\begin{enumerate}[label=(\roman*)]
    \item $q_0\in C^3 (\R)$ and $a<q_0(s)<b$ for every $s \in \R$,
    \item $q_0 \colon \R\to (a,b)$ is increasing and invertible. Moreover $q_0^{-1} \in C^3(a,b)$. \label{q0_inv}
\end{enumerate}
Furthermore, if $W$ satisfies also (W3), there exists a positive constant $C>0$ such that
\begin{enumerate}[label=(\roman*)]\addtocounter{enumi}{2}
    \item $q_0(s)-a\leq Ce^{s/C}$ for every $s \in \R$,
    \item $b-q_0(s)\leq Ce^{-s/C}$ for every $s \in \R$,
    \item $0<\dot{q}_0(s)\leq Ce^{-|s|/C}$ for every $s \in \R$,
    \item $|\ddot{q}_0(s)|\leq Ce^{-|s|/C}$ for every $s \in \R$.
\end{enumerate}
\end{lemma}

\begin{remark}\label{remark:transaltions}
\rm The solutions of the equation (\ref{def:q0}) with initial datum $q(s_0)= \gamma \in (a,b)$ are just the corresponding translations of $q_0$, more precisely
\begin{equation}
    \label{eq:traslati}
    q(s)=q_0\left(s-s_0+q_0^{-1}(\gamma)\right) \quad \forall s \in \R.
\end{equation}
\end{remark}

Now we introduce the perturbation of the one-dimensional optimal profile that is crucial for our proof. We point out that in the following the presence of the parameter $t$ is important for technical reasons, in order to be able to treat the general case. In fact, it is possible to prove Theorem \ref{teo:gammalim} when $E$ is a ball without introducing this additional parameter (and consequently in this case one can ignore the derivatives with respect to $t$).

Let us fix a function $\eta \in C^2 _c (\R^2)$ and consider the following family of ordinary differential equations:
\begin{equation}\label{def:q_eps}
\begin{cases}
\partial_s q_\eps (t,s) = \sqrt{W(q_\eps(t,s)) - \eps \eta(t,s)} & \forall s\in \mathcal{I}_{\eps,t}, \\
q_\eps(t,0)=\frac{a+b}{2}  &\forall t \in \R,
\end{cases}    
\end{equation}
where $\mathcal{I}_{\eps,t}$ is the maximal open interval containing $0$ for which the solution exists.

We observe that the notation is (almost) consistent with the notation for the optimal profile, in the sense that $q_\eps(t,s)=q_0(s)$ when $\eps=0$, and hence $\mathcal{I}_{0,t}=\R$ for every $t\in \R$. In the following lemma we show that $q_\eps$ is globally defined also if $\eps$ is small enough.

\begin{lemma}\label{lemma:I=R2}
Let $\eta \in C^2 _c (\R^2)$ and let $\eps_0, R>0$ be two positive real numbers such that $\supp \eta \subset [-R,R]^2$ and
\begin{equation}\label{def:eps0}
\eps_0 \cdot \max_{(t,s)\in \R^2} |\eta(t,s)| < \min_{|s| \leq R} W(q_0(2s)).
\end{equation}
Then for every $(\eps,t) \in [0,\eps_0) \times \R$ we have $\mathcal{I}_{\eps,t}=\R$.
\end{lemma}

\begin{proof}
We observe that for every $(\eps,t) \in [0,\eps_0) \times \R$ the function $q_0(2s)$ is a supersolution of the equation. Combining this with the fact that $\partial_s q_\eps \geq 0$ in $\mathcal{I}_{\eps,t}$, we deduce
\begin{gather*}
q_0(2s) \leq q_\eps (t,s)\leq \frac{a+b}{2} \quad \forall s\in \mathcal{I}_{\eps ,t} \cap (-\infty,0], \\
\frac{a+b}{2} \leq q_\eps (t,s) \leq q_0 (2s) \quad \forall s \in \mathcal{I}_{\eps ,t}\cap [0,+\infty).
\end{gather*}
Therefore, from (\ref{def:eps0}), we obtain
$$\inf_{s \in \mathcal{I}_{\eps,t} \cap [-R,R]} \Big(W(q_\eps(t,s))-\eps \eta(t,s)\Big) >0,$$
for every $(\eps,t) \in [0,\eps_0)\times \R$. It follows that $[-R,R]\subseteq \mathcal{I}_{\eps,t}$ for every $(\eps,t) \in [0,\eps_0)\times \R$, otherwise the solution could be extended, violating the maximality of $\mathcal{I}_{\eps,t}$. On the other hand (\ref{def:q_eps}) outside $[-R,R]$ reduces to the equation for $q_0$, that has a globally defined solution for every initial datum in $(a,b)$ (see Remark \ref{remark:transaltions}). Hence $q_\eps$ is defined on the whole of $\R^2$ if $\eps\in [0,\eps_0)$.
\end{proof}

In the following lemma, we list the properties of $q_\eps$ that we need in the proof of Theorem~\ref{teo:gammalim}.

\begin{lemma}\label{lemma_ODE}
Let $\eta$ and $\eps_0>0$ be as in Lemma \ref{lemma:I=R2}. Then there exist two positive real numbers $\eps_1 \in (0,\eps_0]$ and $C>0$ depending only on the functions $\eta$ and $W$, such that for every $\eps \in [0,\eps_1)$ the function $q_\eps$ has the following properties:
\begin{enumerate}[label=(A\arabic*)]
    \item $q_\eps \in C^2(\R^2)$ and $a< q_\eps(t,s)< b$ for every $(t,s)\in \R^2$,\label{qeps_in_(a,b)}
    \item $q_\eps(t,s)-a\leq Ce^{s/C}$ for every $(t,s)\in \R^2$, \label{q-a}
    \item $b-q_\eps(t,s)\leq Ce^{-s/C}$ for every $(t,s)\in \R^2$, \label{b-q}
    \item $0<\partial_s q_\eps(t,s)\leq Ce^{-|s|/C}$ for every $(t,s)\in \R^2$, \label{q'<exp}
    \item $\abs{\partial_s^2 q_\eps(t,s)}\leq Ce^{-|s|/C}$ for every $(t,s)\in \R^2$. \label{q''<exp}
\end{enumerate}    
Moreover, as $\eps \to 0^+$,
\begin{enumerate}[label=(B\arabic*)]
    \item $q_\eps(t,s) \to q_0(s)$ uniformly on $\R^2$,\label{qeps-q0}
    \item $\partial_s q_\eps(t,s) \to \dot{q}_0(s)$ uniformly on $\R^2$, \label{qeps'-q0'}
    \item $\partial_t q_\eps(t,s)\to 0$ uniformly on $\R^2$,\label{dtqeps-0}
    \item $\partial_t ^2 q_\eps(t,s)\to 0$ uniformly on $\R^2$.\label{dttqeps-0}
\end{enumerate}
\end{lemma}

\begin{proof}
By the smooth dependence of solutions of ordinary differential equations on parameters (see, for example, \cite{Hartman}, Theorem 4.1 and Corollary 4.1 on pages 100-101) we know that the function $Q(\eps,t,s):=q_\eps(t,s)$ belongs to $C^2([0,\eps_0)\times \R^2)$.

Let $R>0$ be as in Lemma \ref{lemma:I=R2}. Since $Q(0,t,s)=q_0(s)$ for every $(t,s)\in \R^2$, from the regularity of $Q$ we deduce that the convergences in \ref{qeps-q0}-\ref{dttqeps-0} hold uniformly on compact subsets of $\R^2$ and hence, by the properties of $q_0$ listed in Lemma \ref{lemma:q0}, there exist two positive real numbers $\eps_R\in (0,\eps_0]$ and $C_R>0$ such that for every $\eps\in [0,\eps_R)$ all the properties \ref{qeps_in_(a,b)}-\ref{q''<exp} hold with $C=C_R$ if the domain $\R^2$ is replaced by $[-R,R]^2$.

Outside $[-R,R]^2$ the function $\eta$ is identically zero, so $q_\eps$ solves the equation
$$\partial_s q_\eps(t,s)=\sqrt{W(q_\eps(t,s))}\quad \forall (t,s) \in [-R,R] \times [R,+\infty),$$
with initial datum $q_\eps(t,R)\in (a,b)$, at least for $\eps\in [0,\eps_R)$. Therefore, from Remark \ref{remark:transaltions} we deduce that
$$q_\eps(t,s)=q_0\left(s-R+q_0^{-1}(q_\eps(t,R))\right) \qquad \forall (\eps,t,s) \in [0,\eps_R] \times [-R,R] \times [R,+\infty).$$

Differentiating the previous identity, and exploiting the equation for $q_0$, we obtain
$$\partial_s q_\eps (t,s)= \dot{q}_0\left(s-R+q_0^{-1}(q_\eps(t,R))\right),\quad \partial_s ^2 q_\eps (t,s)= \ddot{q}_0\left(s-R+q_0^{-1}(q_\eps(t,R))\right),$$

$$\partial_t q_\eps (t,s)= \dot{q}_0\left(s-R+q_0^{-1}(q_\eps(t,R))\right)\frac{\partial_t q_\eps(t,R)}{\sqrt{W(q_\eps(t,R))}},$$
\begin{multline*}
\partial_t ^2 q_\eps (t,s)= \ddot{q}_0\left(s-R+q_0^{-1}(q_\eps(t,R))\right) \frac{(\partial_t q_\eps(t,R))^2}{W(q_\eps(t,R))} \\[1ex]
+ \frac{\dot{q}_0\left(s-R+q_0^{-1}(q_\eps(t,R))\right)}{\sqrt{W(q_\eps(t,R))}} \left( \partial_t ^2 q_\eps(t,R) - \frac{(\partial_t q_\eps(t,R))^2}{2 W(q_\eps(t,R))} \right),
\end{multline*}
for every $(t,s) \in \R \times [R,+\infty)$.

By the regularity of $Q$, we know that $q_\eps(\cdot,R)$ converges to the constant $q_0(R)$ in $C^2([-R,R])$, and therefore its first and second derivatives converge uniformly to 0 in $[-R,R]$.

Thus, from the previous identities and Lemma \ref{lemma:q0} we deduce that \ref{qeps-q0}-\ref{dttqeps-0} are uniform also in the set $[-R,R]\times [R,+\infty)$, and also \ref{qeps_in_(a,b)}-\ref{q''<exp} hold for every $(t,s)\in [-R,R]\times [R,+\infty)$ if $\eps$ is small enough, possibly changing the values of the constant $C$.

The extension to $[-R,R]\times (-\infty,-R]$ is analogous, hence we have proved all the desired properties in the set $[-R,R]\times \R$, with positive constants $\eps_1$ and $C$, possibly different from $\eps_R$ and $C_R$.

Finally, we observe that $q_\eps(t,s)=q_0(s)$ when $t \notin [-R,R]$, because in this case $\eta(t,s)=0$ for every $s \in \R$ and the equation for $q_\eps$ reduces to the equation for $q_0$. This means that outside $[-R,R]\times \R$ all the properties simply follows from Lemma \ref{lemma:q0}, so the proof is complete.
\end{proof}

We observe that from (\ref{def:q_eps}) we obtain the identity
\begin{equation}\label{ODE^2}
(\partial_s q_\eps (t,s))^2 = W(q_\eps(t,s)) - \eps \eta(t,s)\qquad \forall (\eps,t,s)\in [0,\eps_1)\times \R^2,
\end{equation}
hence differentiating once more with respect to $s$,
\begin{equation}\label{q_eps''}
2\partial_s ^2 q_\eps (t,s) = W'(q_\eps(t,s)) - \eps \frac{\partial_s \eta(t,s)}{\partial_s q_\eps (t,s)} \qquad \forall (\eps,t,s)\in [0,\eps_1)\times \R^2,
\end{equation}
where the denominator never vanishes because of property \ref{q'<exp} in Lemma \ref{lemma_ODE}.

\subsection{Smooth hypersurfaces and the signed distance}

Now we introduce some more notation and we recall some well-known facts about the signed distance function from a smooth hypersurface in $\R^n$.

To this end, let $E\subset \R^n$ be an open set such that its boundary $\Sigma:=\partial E \subset \R^n$ is a closed hypersurface of class $C^\infty$. Let us denote with $\nu_\Sigma \colon \Sigma\to S^{n-1}$ the inner unit normal to $\Sigma$ and with $\dsig \colon \R^n \to \R$ the signed distance function from $\Sigma$, positive inside $E$, that is defined as follows:
\[
    \dsig(x):=\mathrm{dist}(x,\R^n\setminus E)-\mathrm{dist}(x,E)
    =
    \begin{cases}
        \mathrm{dist}(x,\Sigma) & \text{if $x \in E$}, \\
        -\mathrm{dist}(x,\Sigma) & \text{if $x \notin E$}.\\
    \end{cases}
\]

Finally, we denote with $H_\Sigma \colon \Sigma \to \R$ the scalar mean curvature of $\Sigma$, that is $H_\Sigma:=-\div_\Sigma \nu_\Sigma$, where $\div_\Sigma$ denotes the tangential divergence on $\Sigma$.

In the following lemma we state some well-known properties of the signed distance function from a smooth  hypersurface (see for instance \cite{Ambrosio} and Lemma 3 in \cite{Modica}).
\begin{lemma}[Properties of the signed distance in a tubular neighborhood]\label{lemma:dist}
Let $E$, $\Sigma$, $\nu_\Sigma$, $\dsig$ and $H_\Sigma$ be as above. Then there exists a positive real number $\delta>0$ such that the restriction of $\dsig$ to the set $U_\delta(\Sigma):=\{x\in \R^n : |\dsig(x)|\leq \delta\}$ is of class $C^\infty$. Moreover, for every $x\in U_\delta(\Sigma)$ there exists a unique point $\pi_\Sigma(x)\in \Sigma$ such that
\[
    |x-\pi_\Sigma(x)|=\min \{|x-y|:y \in \Sigma\},
\]
and $\pi_\Sigma \colon U_\delta(\Sigma)\to \Sigma$ is smooth.

With these notations, it turns out that
\begin{equation}\label{grad_dist}
\nabla \dsig (x)=\nu_\Sigma(\pi_\Sigma(x))\qquad\forall x\in U_\delta(\Sigma),\end{equation}
\begin{equation}\label{laplacian_dist}
-\Delta \dsig (x)=\sum_{i=1} ^{n-1} \frac{k_i(\pi_\Sigma(x))}{1-\dsig(x)k_i(\pi_\Sigma(x))} \qquad\forall x\in U_\delta(\Sigma),\end{equation}
where $\{k_i(y)\}$ are the principal curvatures of $\Sigma$ at the point $y\in\Sigma$, computed as usual with respect to the outer normal $-\nu_\Sigma(y)$.

In particular,
\begin{equation}\label{nabla_d=1}
|\nabla \dsig (x)|=1 \qquad \forall x \in U_\delta(\Sigma),
\end{equation}
\begin{equation}\label{Delta_d=H}
-\Delta \dsig (y)= H_\Sigma (y) \qquad \forall y \in \Sigma.
\end{equation}

Finally,
\begin{equation}\label{lim_meas_d=t}
\lim_{t\to 0} \mathcal{H}^{n-1}(\{\dsig =t\}) = \mathcal{H}^{n-1}(\Sigma),
\end{equation}
and therefore (by the coarea formula, (\ref{nabla_d=1}) and (\ref{lim_meas_d=t}))
\begin{equation}\label{meas_r}
\lim_{r\to 0^+} \frac{\Leb^n(\{|\dsig(x)|\leq r\})}{2r}=\mathcal{H}^{n-1}(\Sigma).
\end{equation}

\end{lemma}

\section{Proof of Theorem \ref{teo:gammalim}}

This section is entirely devoted to the proof of Theorem \ref{teo:gammalim}.

First of all, we observe that by a standard density argument, we can assume that $\Sigma:=\partial E$ is a closed hypersurface of class $C^\infty$. Indeed, if $E$ is a set of finite perimeter in $\R^n$, then either $E$ or its complement has finite measure, thus it can be approximated (in the $L^1$ topology) by bounded smooth sets in such a way that the perimeter of the approximating sets converges to the perimeter of the initial one.

For every positive real number $\eps>0$ let us fix a function $\vartheta_\eps \in C^\infty(\R; [0,1])$ with the following properties:
\[
    \vartheta_\eps(s)=
    \begin{cases}
    0 & \text{if $s \in \left[-\frac{1}{\sqrt{\eps}},\frac{1}{\sqrt{\eps}}\right] $},\\[1.5ex]
    1 & \text{if $s \in \left(-\infty,-\frac{1}{\sqrt{\eps}}-1\right] \cup \left[\frac{1}{\sqrt{\eps}}+1,+\infty\right)$},
\end{cases}
\]
    \begin{equation}
        \label{theta'<2}
        |\dot{\vartheta}_\eps (s)| \leq 2 \quad \text{and} \quad |\ddot{\vartheta}_\eps (s)| \leq 8 \quad \forall s \in \R.
    \end{equation}

Let $U_\delta(\Sigma)$ and $\pi_\Sigma$ be as in Lemma \ref{lemma:dist} and consider the map $\Psi_\eps \colon U_\delta(\Sigma)\to \R \times [-\delta/\eps,\delta/\eps]$ defined as
    \[
        \Psi_\eps(x)=\left(h(x), \frac{\dsig(x)}{\eps}\right),
    \]
where
    \[
        h(x):=H_\Sigma(\pi_\Sigma(x))
    \]
and $H_\Sigma (y)\in \R$ is the mean curvature of $\Sigma$ at the point $y\in \Sigma$.

We observe that $h \colon U_\delta(\Sigma) \to \R$ is smooth and it is constant in the direction normal to $\Sigma$. Hence all the derivatives of $h$ are uniformly bounded in $U_\delta(\Sigma)$ and
    \begin{equation}
        \label{nablah_perp_nablad}
        \nabla h(x) \cdot \nabla \dsig(x) =0 \qquad \forall x \in U_\delta(\Sigma).    
    \end{equation}

We point out that if $\Sigma$ is a sphere then $h$ is constant, so only the second component of the function $\Psi_\eps$ varies.

We set also
\begin{equation}\label{def:A_eps}
A_\eps:=\left\{x \in \R^n : |\dsig(x)| < \sqrt{\eps} \right\}, \ \ B_\eps:=\left\{x \in \R^n : \sqrt{\eps}\leq |\dsig(x)| \leq \sqrt{\eps}+\eps \right\},    
\end{equation}
and we notice that (\ref{meas_r}) implies that
\begin{equation}\label{meas_Aeps_Beps}
\lim_{\eps \to 0^+} \Leb^n(A_\eps)=\lim_{\eps \to 0^+} \Leb^n(B_\eps)=0.
\end{equation}

Let us fix $\eta \in C^2 _c (\R^2)$ and consider the functions $q_\eps$ as in (\ref{def:q_eps}). Let $\eps_1>0$ be as in Lemma~\ref{lemma_ODE} and for every $\eps \in (0,\eps_1)$ such that
\begin{equation}\label{eps<<delta}
\sqrt{\eps}+\eps\in (0,\delta),\end{equation}
let us define the function $u_\eps:\R^n \to \R$ as follows:
    \[
        u_\eps(x):= q_\eps(\Psi_\eps(x))\left(1-\vartheta_\eps\left(\frac{\dsig(x)}{\eps}\right)\right) + \chi_E ^{a,b}(x) \vartheta_\eps\left(\frac{\dsig(x)}{\eps}\right).
    \]

We observe that, despite $\Psi_\eps$ is defined only in $U_\delta (\Sigma)$, the function $u_\eps$ is well-defined on $\R^n$ and of class $C^2$. Indeed, we have $\vartheta_\eps(\dsig(x)/\eps)\neq 1$ only if $x \in A_\eps \cup B_\eps$, and (\ref{eps<<delta}) ensures that $A_\eps \cup B_\eps\subseteq U_\delta(\Sigma)$, so $\Psi_\eps(x)$ is well-defined whenever $1-\vartheta(\dsig(x)/\eps)\neq 0$. Moreover, $u_\eps \in C^2(\R^n)$ because $q_\eps \in C^2(\R^2)$, the functions $\Psi_\eps$ and $\vartheta_\eps(\dsig/\eps)$ are smooth and $\vartheta_\eps(\dsig/\eps)=0$ in $A_\eps$, while $\chi_E ^{a,b}$ is smooth outside $A_\eps$.

We observe also that $u_\eps(x)=\chi^{a,b}_E (x)$ for every $x\notin A_\eps\cup B_\eps$. Moreover from property \ref{qeps_in_(a,b)} in Lemma \ref{lemma_ODE} we have that $u_\eps(x) \in (a,b)$ for every $x \in A_\eps \cup B_\eps$. This implies that
$$\lim_{\eps \to 0^+}\left\|u_\eps - \chi^{a,b}_E\right\|_{L^1(\R^n)}\leq \lim_{\eps \to 0^+} 2(b-a)(\Leb^n(A_\eps) +\Leb^n(B_\eps))=0.$$

Now we compute the gradient of $u_\eps$ and we find
\begin{align*}
    \nabla u_\eps(x)  = & 
    \left[  \partial_t q_\eps(\Psi_\eps(x)) \nabla h(x) + \partial_s q_\eps(\Psi_\eps(x)) \frac{\nabla \dsig(x)}{\eps} \right]\left(1-\vartheta_\eps\left(\frac{\dsig(x)}{\eps}\right)\right) \\
    & + 
    \left[\chi_E ^{a,b}(x)-q_\eps(\Psi_\eps(x))\right] \dot{\vartheta}_\eps \left(\frac{\dsig(x)}{\eps}\right)\frac{\nabla \dsig (x)}{\eps}.
\end{align*}

In order to compute the Laplacian of $u_\eps$, let us set
$$V_\eps(x):= \div \big(\partial_t q_\eps(\Psi_\eps(x)) \nabla h(x)\big) =\partial_t ^2 q_\eps(\Psi_\eps(x)) |\nabla h(x)|^2 + \partial_t q_\eps(\Psi_\eps(x)) \Delta h(x),$$
where the second identity follows from (\ref{nablah_perp_nablad}).

Therefore, recalling (\ref{nabla_d=1}) and exploiting again (\ref{nablah_perp_nablad}), we have
\begin{align}
    \Delta u_\eps(x) = &
    \left[V_\eps+ \partial_s^2 q_\eps(\Psi_\eps) \frac{1}{\eps^2} +\partial_s q_\eps(\Psi_\eps) \frac{\Delta \dsig}{\eps} \right] \left(1-\vartheta_\eps \left(\frac{\dsig}{\eps}\right)\right)\nonumber \\
    & -
    2\dot{\vartheta}_\eps \left(\frac{\dsig}{\eps}\right) \partial_s q_\eps(\Psi_\eps) \frac{1}{\eps^2} \nonumber\\
    & +
    \left[\chi_E ^{a,b}-q_\eps(\Psi_\eps)\right]\left[\ddot{\vartheta}_\eps \left(\frac{\dsig}{\eps}\right)\frac{1}{\eps^2}+\dot{\vartheta}_\eps \left(\frac{\dsig}{\eps}\right)\frac{\Delta \dsig}{\eps}\right],\label{Deltaueps}
\end{align}
where every term in the right-hand side is computed at $x$.

Now we compute the functionals, starting with $E_\eps$. First of all, recalling the definitions of $A_\eps$ and $B_\eps$ in (\ref{def:A_eps}), we observe that
$$E_\eps(u_\eps,\R^n)=E_\eps(u_\eps,A_\eps)+E_\eps(u_\eps,B_\eps),$$
because $u_\eps=\chi^{a,b}_E$ elsewhere.

We claim that
\begin{equation}\label{energy_near_Sigma}
\lim_{\eps\to 0^+} E_\eps(u_\eps,A_\eps)=\sigma_W ^{a,b} \mathcal{H}^{n-1}(\Sigma),
\end{equation}
and
\begin{equation}\label{energy_far}
\lim_{\eps\to 0^+} E_\eps(u_\eps,B_\eps)=0,
\end{equation}
independently of the choice of $\eta$.

\smallskip

\textsc{Proof of (\ref{energy_near_Sigma})}: We observe that on $A_\eps$
\begin{align}
 \eps |\nabla u_\eps|^2 + \frac{W(u_\eps)}{\eps}&= \eps\left|  \partial_t q_\eps(\Psi_\eps) \nabla h+ \partial_s q_\eps(\Psi_\eps) \frac{\nabla \dsig}{\eps} \right|^2 +\frac{W(q_\eps(\Psi_\eps))}{\eps}\nonumber\\
&=\eps (\partial_t q_\eps(\Psi_\eps))^2 |\nabla h|^2 +\frac{1}{\eps}\left[(\partial_s q_\eps(\Psi_\eps))^2 +W(q_\eps(\Psi_\eps))\right],\label{energy_density_near_Sigma}
\end{align}
because the double product in the square vanishes thanks to (\ref{nablah_perp_nablad}).

From \ref{dtqeps-0} in Lemma \ref{lemma_ODE} and (\ref{meas_Aeps_Beps}) we deduce that the integral on $A_\eps$ of the first addendum in the right-hand side of (\ref{energy_density_near_Sigma}) goes to zero, and hence
\begin{align*}
\lim_{\eps\to 0^+} E_\eps(u_\eps,A_\eps)=\lim_{\eps\to 0^+} \int_{A_\eps} \frac{(\partial_s q_\eps(\Psi_\eps))^2 +W(q_\eps(\Psi_\eps))}{\eps}\, dx.
\end{align*}

Now we compute the integral in the right-hand side exploiting the coarea formula and (\ref{nabla_d=1}), so we obtain
\begin{align*}
&\int_{A_\eps}\frac{(\partial_s q_\eps(\Psi_\eps))^2 +W(q_\eps(\Psi_\eps))}{\eps}\, dx\\
=&\int_{-\frac{1}{\sqrt{\eps}}} ^{\frac{1}{\sqrt{\eps}}} ds \int_{\{\dsig=\eps s\}} \left[(\partial_s q_\eps (h(y),s))^2 + W(q_\eps(h(y),s))\right] d\mathcal{H}^{n-1}(y).
\end{align*}

By (\ref{lim_meas_d=t}) and properties \ref{q-a}, \ref{b-q}, \ref{q'<exp}, \ref{qeps-q0} and \ref{qeps'-q0'} in Lemma \ref{lemma_ODE} we can pass to the limit and obtain
$$ \lim_{\eps\to 0^+} E_\eps(u_\eps, A_\eps)= \mathcal{H}^{n-1}(\Sigma)\int_{-\infty} ^{+\infty} \left[\dot{q}_0(s)^2 + W(q_0(s))\right] ds= \sigma_W ^{a,b} \mathcal{H}^{n-1}(\Sigma),$$
that is exactly (\ref{energy_near_Sigma}).

\smallskip

\textsc{Proof of (\ref{energy_far})}: We observe that on $B_\eps$
\[
    \eps |\nabla u_\eps|^2 + \frac{W(u_\eps)}{\eps} 
    \leq   
    2\eps (\partial_t q_\eps(\Psi_\eps))^2 |\nabla h|^2 +\frac{2}{\eps}(\partial_s q_\eps(\Psi_\eps))^2  + 
    2\left[\chi_E ^{a,b}-q_\eps(\Psi_\eps)\right]^2 \frac{4}{\eps} +\frac{W(u_\eps)}{\eps},
\]
where we exploited again (\ref{nabla_d=1}), (\ref{theta'<2}) and (\ref{nablah_perp_nablad}), as well as the inequality $(\alpha+\beta)^2 \leq 2\alpha^2 + 2\beta^2$.

As before, the first term is uniformly bounded (actually vanishing as $\eps \to 0^+$) and hence its integral on $B_\eps$ goes to zero.
For the term with $W$ we observe that, if $x \in B_\eps$, then
\begin{equation}\label{eq:u-chi<q-chi}
|u_\eps(x)-\chi_E ^{a,b}(x)|\leq |q_\eps(\Psi_\eps(x))-\chi_E ^{a,b}(x)|,
\end{equation}
hence from the Lipschitz continuity of $W$ in $[a,b]$, properties \ref{q-a} and \ref{b-q} in Lemma \ref{lemma_ODE} and the fact that $|\dsig(x)|/\eps\geq 1/\sqrt{\eps}$ for every $x \in B_\eps$ we deduce that in $B_\eps$ we have
\[
    \frac{W(u_\eps)}{\eps}\leq \Lip(W,[a,b]) \cdot \frac{|q_\eps(\Psi_\eps)-\chi_E ^{a,b}|}{\eps} \leq \Lip(W,[a,b]) \cdot \frac{Ce^{-1/(C\sqrt{\eps}) }}{\eps}.
\]

In a similar way, exploiting again Lemma \ref{lemma_ODE}, we obtain
$$\frac{2}{\eps}(\partial_s q_\eps(\Psi_\eps))^2\leq \frac{2C^2e^{-2/(C\sqrt{\eps})}}{\eps},$$
and
$$\frac{8}{\eps}\left[\chi_E ^{a,b}-q_\eps(\Psi_\eps)\right]^2 \leq \frac{8C^2e^{-2/(C\sqrt{\eps})}}{\eps}.$$

Therefore the integrand is infinitesimal in the set $B_\eps$, and this is enough to establish (\ref{energy_far}).

\smallskip

Now we focus on $G_\eps(u_\eps,\R^n)$ and, as before, we observe that
$$G_\eps(u_\eps,\R^n)=G_\eps(u_\eps,A_\eps)+G_\eps(u_\eps,B_\eps).$$

We claim that
\begin{equation}\label{grad_near_Sigma}
\limsup_{\eps\to 0^+} G_\eps(u_\eps,A_\eps)
\leq 4\int_{-\infty} ^{+\infty} ds \int_{\Sigma} \left[\partial_s \eta(H(y),s)+2\dot{q}_0(s)^2H(y) \right]^2 d\mathcal{H}^{n-1}(y),
\end{equation}
and
\begin{equation}\label{grad_far}
\lim_{\eps\to 0^+} G_\eps(u_\eps,B_\eps)= 0.
\end{equation}

\smallskip

\textsc{Proof of (\ref{grad_near_Sigma})}:
In the set $A_{\eps}$, recalling~(\ref{Deltaueps}) and exploiting again the elementary inequality $(\alpha+\beta)^2 \leq 2\alpha^2 +2\beta^2$, we obtain
\begin{align*}
    \left(2\eps\Delta u_\eps -\frac{W'(u_\eps)}{\eps}\right)^2
    &
    \leq 2\left[2\eps V_\eps \right]^2 + 2 \left[\frac{2\partial_s^2 q_\eps(\Psi_\eps)-W'(q_\eps(\Psi_\eps))}{\eps} +2\partial_s q_\eps(\Psi_\eps)\Delta \dsig \right]^2 \\
    &
    = 8\eps^2 V_\eps^2 +2\left[2\partial_s q_\eps(\Psi_\eps)\Delta \dsig - \frac{\partial_s \eta(\Psi_\eps)}{\partial_s q_\eps (\Psi_\eps)}\right]^2,
\end{align*}
where the equality follows from (\ref{q_eps''}).

Since $V_\eps$ is bounded because of the smoothness of $h$ and properties \ref{dtqeps-0} and \ref{dttqeps-0} in Lemma~\ref{lemma_ODE}, from (\ref{energy_near_Sigma}) we deduce
$$\lim_{\eps\to 0^+} \int_{A_\eps} 8\eps^2 V_\eps^2\left(\eps|\nabla u_\eps|^2+\frac{W(u_\eps)}{\eps}\right) dx =0.$$

Moreover, from (\ref{meas_Aeps_Beps}) we get
$$\lim_{\eps\to 0^+} \int_{A_\eps} \eps (\partial_t q_\eps(\Psi_\eps))^2 |\nabla h|^2 \left[2\partial_s q_\eps(\Psi_\eps)\Delta \dsig - \frac{\partial_s \eta(\Psi_\eps)}{\partial_s q_\eps (\Psi_\eps)}\right]^2=0,$$
because every function in the integral is uniformly bounded as $\eps\to 0^+$ thanks to Lemma \ref{lemma_ODE}. We point out that also the last term is bounded, because $\dot{q}_0(s)$ is larger than a positive constant in the support of $\eta$ and $\partial_s q_\eps\to \dot{q}_0$ uniformly by \ref{qeps'-q0'} in Lemma \ref{lemma_ODE}.

Therefore, combining the previous limits with (\ref{energy_density_near_Sigma}), we obtain
\begin{align*}
    \limsup_{\eps\to 0^+} G_\eps(u_\eps,A_\eps)
    &\leq 
    \limsup_{\eps\to 0^+} \int_{A_\eps} 2\left[2\partial_s q_\eps(\Psi_\eps)\Delta \dsig - \frac{\partial_s \eta(\Psi_\eps)}{\partial_s q_\eps (\Psi_\eps)}\right]^2 \cdot \frac{(\partial_s q_\eps(\Psi_\eps))^2 +W(q_\eps(\Psi_\eps))}{\eps}\,dx \\
    &= 
    \limsup_{\eps\to 0^+} \left\{ \int_{A_\eps} 2\left[2\partial_s q_\eps(\Psi_\eps)\Delta \dsig - \frac{\partial_s \eta(\Psi_\eps)}{\partial_s q_\eps (\Psi_\eps)}\right]^2 \frac{2(\partial_s q_\eps(\Psi_\eps))^2}{\eps}\,dx\right.\\
    & \qquad \qquad \left.+\int_{A_\eps} 2\left[2\partial_s q_\eps(\Psi_\eps)\Delta \dsig - \frac{\partial_s \eta(\Psi_\eps)}{\partial_s q_\eps (\Psi_\eps)}\right]^2 \eta(\Psi_\eps)\,dx\right\},
\end{align*}
where the equality follows from (\ref{ODE^2}).

We observe that the integral in the last line goes to zero as $\eps \to 0^+$, because again all functions in the integral are bounded and (\ref{meas_Aeps_Beps}) holds.

So, in the end,
$$\limsup_{\eps\to 0^+} G_\eps(u_\eps, A_\eps) \leq \limsup_{\eps\to 0^+} \int_{A_\eps} \frac{4}{\eps}\left[2(\partial_s q_\eps(\Psi_\eps))^2\Delta \dsig - \partial_s \eta(\Psi_\eps)\right]^2 \,dx.$$

Now we apply the coarea formula exploiting again (\ref{nabla_d=1}) and we get
\begin{align*}
&\int_{A_\eps} \frac{4}{\eps}\left[2(\partial_s q_\eps(\Psi_\eps))^2\Delta \dsig - \partial_s \eta(\Psi_\eps)\right]^2 \,dx\\
=&\int_{-\frac{1}{\sqrt{\eps}}} ^{\frac{1}{\sqrt{\eps}}} ds \int_{\{\dsig=\eps s\}}4\left[2(\partial_s q_\eps(h(y),s))^2\Delta \dsig(y) - \partial_s \eta(h(y),s)\right]^2 d\mathcal{H}^{n-1}(y).
\end{align*}

We can pass to the limit exploiting again Lemma \ref{lemma_ODE} and, recalling (\ref{Delta_d=H}) and (\ref{lim_meas_d=t}), we obtain (\ref{grad_near_Sigma}).

\smallskip

\textsc{Proof of (\ref{grad_far})}: In the set $B_\eps$ we use again the inequality $(\alpha+\beta)^2\leq 2\alpha^2 +2\beta^2 $, so we have
\begin{align*}
\left(2\eps \Delta u_\eps -\frac{W'(u_\eps)}{\eps} \right)^2\leq 8\eps^2(\Delta u_\eps) ^2 + \frac{2}{\eps^2}W'(u_\eps)^2.
\end{align*}

Now we look at the expression (\ref{Deltaueps}) for the Laplacian of $u_\eps$: we recall that $V_\eps$ is bounded, and we observe that each other term is controlled by terms of the form $C\eps^{-2} e^{-1/(C\sqrt{\eps})}$ because of properties \ref{q-a}-\ref{q''<exp} in Lemma \ref{lemma_ODE} and the fact that $|\dsig(x)|/\eps\geq 1/\sqrt{\eps}$ for every $x\in B_\eps$.

Concerning the term with $W'(u_\eps)$, it is enough to recall (\ref{eq:u-chi<q-chi}), hence the Lipschitz continuity of $W'$ on $[a,b]$ implies that in $B_\eps$ we have
\[
    |W'(u_\eps)|\leq \mathrm{Lip}(W',[a,b]) \cdot |q_\eps(\Psi_\eps)-\chi^{a,b}_E|,
\]
and again the right-hand side is bounded by $Ce^{-1/(C\sqrt{\eps})}$ thanks to conditions \ref{q-a} and \ref{b-q} in Lemma \ref{lemma_ODE}.

Therefore $\left(2\eps \Delta u_\eps - W'(u_\eps)/\eps \right)^2$ is uniformly bounded (and actually infinitesimal as $\eps \to 0^+$) in the set $B_\eps$ and this, together with (\ref{energy_far}), implies (\ref{grad_far}).

At this point, the conclusion follows from the following lemma and a diagonal argument.

\begin{lemma}\label{lemma:infF=0}
Let $\Sigma\subset \R^n$ be a closed hypersurface of class $C^\infty$ and let $H:\Sigma\to\R$ be a continuous function. Let us set
$$F(\eta):=\int_{-\infty} ^{+\infty} ds \int_{\Sigma} \left[\partial_s \eta(H(y),s)+2\dot{q}_0(s)^2H(y) \right]^2 d\mathcal{H}^{n-1}(y) \qquad \forall \eta \in C^1(\R^2).$$
Then
$$\inf \left\{F(\eta) :\eta \in C^2 _c(\R^2) \right\}=0.$$
\end{lemma}

\begin{proof}
Let us consider the function $\Phi \colon \R^2\to \R$ defined as
$$\Phi(t,s):=-2t\int_{-\infty} ^s \dot{q}_0(\tau)^2 \,d\tau.$$

We observe that trivially $F(\Phi)=0$, but unfortunately $\Phi$ is not compactly supported, so we need to introduce some cut-off functions.

Let $\rho \in C^\infty _c (\R)$ be a nonnegative smooth function such that
$$\rho(s)=\begin{cases}
1 &\mbox{if } |s|\leq 1, \\
0 &\mbox{if } |s|\geq 2,
\end{cases}$$
and $\dot{\rho}(s)\leq 2$ for every $s \in \R$. For every positive real number $L>0$ we set $\rho_L(s):=\rho(s/L)$ and we observe that $\dot{\rho}_L(s)\leq 2/L$ for every $s\in \R$.

Let also $g\in C^\infty _c(\R)$ be another nonnegative smooth function such that $g(t)=1$ whenever $|t|\leq \max\{|H(y)|:y \in \Sigma\}$.

Now we set
$$\eta_L(t,s):=\Phi(t,s)\rho_L(s)g(t),$$
and we claim that
\begin{equation}\label{F(eta_L)}
\lim_{L\to +\infty} F(\eta_L)=0.
\end{equation}
Since $\eta_L \in C^2 _c(\R^2)$ for every $L>0$, this is enough to conclude the proof.

To prove (\ref{F(eta_L)}) we first observe that
\begin{equation}\label{eq:ds_etaL}
\partial_s \eta_L(t,s)=-2t \dot{q}_0(s)^2 \rho_L(s)g(t)+ \Phi(t,s)\dot{\rho}_L(s)g(t).
\end{equation}

Moreover, $g(H(y))=1$ for every $y \in \Sigma$, so we can ignore the presence of $g$ in the computation of $F(\eta_L)$. Since $\rho_L=1$ in $[-L,L]$, it follows that $\partial_s \eta_L(H(y),s)=-2\dot{q}_0(s)^2H(y)$ for every $s \in [-L,L]$. Hence
$$\int_{\{|s|<L\}} ds \int_{\Sigma} \left[\partial_s \eta_L(H(y),s)+2\dot{q}_0(s)^2H(y) \right]^2 d\mathcal{H}^{n-1}(y)=0.$$

Therefore, from (\ref{eq:ds_etaL}), we have
$$ F(\eta_L)= \int_{\{L \le |s| \le 2L\}} ds \int_{\Sigma} \left[2\dot{q}_0(s)^2 H(y)(1-\rho_L(s))+\Phi(H(y),s)\dot{\rho}_L(s)\right]^2 d\mathcal{H}^{n-1}(y).$$

Exploiting Lemma \ref{lemma:q0}, the properties of $\rho_L$ and the identity
$$\int_{-\infty} ^{+\infty} \dot{q}_0(s)^2\, ds =\int_{-\infty} ^{+\infty} \sqrt{W(q_0(s))}\dot{q}_0(s)\, ds =\int_a ^b \sqrt{W(q)} \, dq =\frac{\sigma^{a,b} _W}{2},$$
we finally obtain
$$F(\eta_L)\leq 2L \int_{\Sigma} 4H(y)^2 \left[C^2e^{-2L/C}+\frac{\sigma_W ^{a,b}}{L} \right]^2 d\mathcal{H}^{n-1}(y),$$
from which we deduce (\ref{F(eta_L)}).
\end{proof}

\begin{remark}\rm
Repeating the same computations with the modified functional $\widehat{G}_\eps$ instead of $G_\eps$, one obtains analogous estimates, but with $F(\eta)$ replaced by
$$\widehat{F}(\eta):=\int_{-\infty} ^{+\infty} ds \int_{\Sigma} \left[\frac{\partial_s \eta(H(y),s)}{\dot{q}_0(s)}+2\dot{q}_0(s) H(y) \right]^2 d\mathcal{H}^{n-1}(y).$$

Since $\eta$ is compactly supported we have
$$\int_{\Sigma} d\mathcal{H}^{n-1}(y) \int_{-\infty} ^{+\infty} 4\partial_s \eta(H(y),s)H(y)=0.$$

Therefore, expanding the square in the definition of $\widehat{F}$ we obtain
\begin{align*}
\widehat{F}(\eta)&= \int_{\Sigma} d\mathcal{H}^{n-1}(y) \int_{-\infty} ^{+\infty} \left[ \frac{(\partial_s \eta(H(y),s))^2}{\dot{q}_0(s)^2}+4\partial_s \eta(H(y),s)H(y) +4\dot{q}_0(s)^2 H(y)^2 \right]ds\\
&=\int_{\Sigma} d\mathcal{H}^{n-1}(y) \int_{-\infty} ^{+\infty} \left[ \frac{(\partial_s \eta(H(y),s))^2}{\dot{q}_0(s)^2} +4\dot{q}_0(s)^2 H(y)^2 \right]ds.
\end{align*}
This shows that in this case our construction can not decrease the energy with respect to the ``standard'' approximation, corresponding to $\eta\equiv 0$.

Furthermore, it can be checked that if $\eta_L$ are the functions defined in the proof of Lemma~\ref{lemma:infF=0} then
$$\lim_{L\to +\infty} \widehat{F}(\eta_L)=+\infty.$$
\end{remark}

\section{An integrality result in radial simmetry}
\label{section:blow-up}

The aim of this section is to give a proof of Theorem~\ref{theorem:main2}. Here we deal mainly with radial functions, hence we recall some definitions and we write the expression of the functionals in this context.

\begin{definition}
We say that $u \colon \mathbb{R}^n \to \mathbb{R}$ is a radial function if there exists a function $\hat{u} \colon [0,+\infty) \to \mathbb{R}$ such that $u(x) = \hat{u}(\abs{x})$ for every $x \in \mathbb{R}^n$.
\end{definition}

In the sequel, we denote $u$ and $\hat{u}$ with the same letter $u$. In particular, if $r \in [0,+\infty)$ is a nonnegative real number, we write $u(r)$ meaning $u(x)$ for some $x \in \mathbb{R}^n$ with $\abs{x} = r$.

We recall that if $u$ is a sufficiently regular radial function then $\Delta u (x) = \ddot{u}(r) + (n-1)\dot{u}(r)/r$, where $r=\abs{x}$. In particular, if $\Omega = \{ c < \abs{x} < d \}$ then we can rewrite the functionals $E_{\eps}(u, \Omega)$ and $G_{\eps}(u, \Omega)$ as follows:
    \begin{gather*}
        E_{\eps}(u, \Omega) = \int_{c}^{d}
        \left( \eps \dot{u}^2 + \frac{W(u)}{\eps} \right)
        \omega_{n-1} r^{n-1} dr, \\[1ex]
        G_{\eps}(u, \Omega) = \int_{c}^{d}
        \left( 2 \eps \ddot{u} + \frac{2 \eps (n-1)}{r} \dot{u} - \frac{W'(u)}{\eps} \right)^2
        \left( \eps \dot{u}^2 + \frac{W(u)}{\eps} \right)
        \omega_{n-1} r^{n-1} dr.
    \end{gather*}
    
The main idea behind the proof of Theorem~\ref{theorem:main2} is to look at the behavior of $u_{\eps}$ near its $\gamma$-level set, where $\gamma$ is any real number in $(a,b)$. To this end, we study the behavior of a suitable rescaling of $u_{\eps}$ that is usually called blow-up. The precise definition is the following.

\begin{definition}
    \label{def:blow-up}
    Let $u_{\eps} \colon \mathbb{R}^n \to \mathbb{R}$ be a family of radial functions and let $\{ R_{\eps} \} \subset (0,+\infty)$ be a family of positive real numbers. The blow-up of $u_{\eps}$ at $R_{\eps}$ is the family of functions $\psi_{\eps} \colon [-R_{\eps}/\eps,+\infty) \to \mathbb{R}$ defined as
    \[
    \psi_{\eps}(s) := u_{\eps}(R_{\eps} + \eps s).
    \]
\end{definition}

The following proposition shows that if $\{u_{\eps}\}$ is a family of functions with equibounded functionals then the blow-up of $u_{\eps}$ at some points $R_{\eps} \ge r_0 > 0$ must subconverge to a translation of the optimal one-dimensional profile $q_0$ if we know in addition that $\{ u_{\eps}(R_{\eps}) \} \subset [a+\delta,b-\delta]$ for some $\delta > 0$. This proposition is the main ingredient in the proof of Theorem~\ref{theorem:main2}.

\begin{proposition}
    \label{prop:blowup}
    Let $W \colon \mathbb{R} \to [0,+\infty)$ be a potential satisfying (W1), (W2) and (W3+). 
    Let $\{ u_\eps \} \subset W^{2,1}_{\mathsf{loc}}(\mathbb{R}^n) \cap W_{\mathsf{loc}}^{1,2}(\R^n)$ be a family of radial functions such that
    \[
        \limsup_{\eps\to 0^+} \Big( E_{\eps}(u_{\eps}, \mathbb{R}^n) + G_{\eps}(u_\eps, \mathbb{R}^n) \Big) < +\infty.
    \]
    Let us assume that there exist two constants $r_0 > 0$ and $\gamma \in (a,b)$ and a family of points $\{ R_{\eps} \} \subset [r_0, +\infty)$ such that $u_{\eps}(R_{\eps}) \to \gamma$ as $\eps \to 0^+$.
    
    Let $\psi_{\eps}$ be the blow-up of $u_{\eps}$ at $R_{\eps}$. Then, for any sequence $\eps_k \to 0^+$ there exists a sequence $\{ m_k \}$ of positive real numbers such that $m_k \to +\infty$, $m_k \eps_k \to 0$ as $k \to +\infty$ and, up to (not relabelled) subsequences, one of the following holds:
\begin{equation}\label{eq:H2conv}
        \text{either} \quad \norm{\psi_{\eps_k} - q_0 \circ \tau_{\gamma}^{+}}_{W^{2,2}(-m_k,m_k)} \to 0
        \quad \text{or} \quad
        \norm{\psi_{\eps_k} - q_0 \circ \tau_{\gamma}^{-}}_{W^{2,2}(-m_k,m_k)} \to 0,    
\end{equation}
    as $k \to +\infty$, where $\tau_{\gamma}^{+}(s) = s + q_0^{-1}(\gamma)$ and $\tau_{\gamma}^{-}(s) = -s + q_0^{-1}(\gamma)$. 
    
    Moreover, in both cases
    \begin{equation}
    \label{eq:state-limit}
        \lim_{k \to +\infty} \int_{-m_k}^{m_k} \Big( \dot{\psi}_{\eps_k}^2 + W(\psi_{\eps_k}) \Big) ds = \int_{-\infty}^{+\infty} \Big( \dot{q}_0^2 + W(q_0) \Big) ds=\sigma^{a,b}_W.
    \end{equation}
\end{proposition}

\begin{remark}\rm
    \label{remark:H2-C1}
    It is well known that if $I \subset \mathbb{R}$ is an open interval then $W^{1,2}(I) \subset C_{\mathsf{b}}(I)$ with continuous injection, where $C_{\mathsf{b}}(I)$ is the space of continuous and bounded functions on $I$. Moreover, for any $u \in W^{1,2}(I)$, it holds (see for instance Theorem~8.8 and footnote 6 on page~209 in \cite{Brezis})
    \begin{equation*}
        \norm{u}_{L^\infty(I)} \le 4 \sqrt{2} \left( 1 + \frac{1}{\Leb^1(I)} \right) \norm{u}_{W^{1,2}(I)}.
    \end{equation*}
    
    In particular, this implies that the conclusion of Proposition~\ref{prop:blowup} holds also if in (\ref{eq:H2conv}) we replace the $W^{2,2}$ norm with the $C^1$ norm.
\end{remark}

In the proof of Proposition~\ref{prop:blowup} we need the following two lemmas.

\begin{lemma}
    \label{lemma:quadrato}
    Let $[c,d] \subset \mathbb{R}$ be an interval and let $\{ f_{\eps} \} \subset C([c,d])$ be a family of continuous functions. Assume that:
    \begin{enumerate}[label=(\roman*)]
        \item $\left\{ f_{\eps}^2 \right\} \subset W^{1,1}(c,d)$ and there exists a constant $M > 0$ such that $\norm*{ f_{\eps}^2 }_{W^{1,1}(c,d)} \le M$,\label{hyp:lemma-quadrato-1}
        \item $f_{\eps} \rightharpoonup f_0$ weakly in $L^1(c,d)$.\label{hyp:lemma-quadrato-2} 
        \end{enumerate}
    Then $f_{\eps} \to f_0$ strongly in $L^p(c,d)$ for any $p \in [1,+\infty)$.
\end{lemma}

\begin{proof}
The proof relies on the following general fact (see~\cite{Brezis}, exercise~4.16(3) on page~123, solution on pages~398-399).

Let $1 \le p < q$. Then:
    \[
        \begin{cases*}
            f_{\eps} \rightharpoonup f_0 \ \text{weakly in $L^q(c,d)$} \\
            f_{\eps} \to f_0 \ \text{a.e. on $(c,d)$} 
        \end{cases*}
            \implies
            f_{\eps} \to f_0 \ \text{strongly in $L^p(c,d)$}.
    \]

Fix $1 \le p < q$; if we prove that $\{ f_{\eps} \}$ is a bounded family in $L^q(c,d)$ then assumption~(ii) immediately implies that $\{f_{\eps}\}$ converges weakly to $f_0$ in $L^q(c,d)$. The boundedness of $\{ f_{\eps}\}$ is an easy consequence of assumption~(i), indeed
\[
    \norm*{f_{\eps}}_{L^q(c,d)}^2 = \norm*{f_{\eps}^2}_{L^{q/2}(c,d)} \le C \norm*{f_{\eps}^2}_{W^{1,1}(c,d)} \le C M,
\]
where $C > 0$ is a constant depending only on the length of $(c,d)$.

Now we prove the almost everywhere convergence up to subsequences. Let us define
\[
    f_{\eps, +} = \max\{f_{\eps},0\}, \quad f_{\eps, -} = \max\{-f_{\eps},0\}.
\]
\textsc{Claim}: $(f_{\eps,+})^2 \in W^{1,1}(c,d)$ and $\norm*{(f_{\eps,+})^2}_{W^{1,1}(c,d)} \le \norm*{f_{\eps}^2}_{W^{1,1}(c,d)}$. A similar statement holds for $(f_{\eps,-})^2$. 

Assuming the claim the conclusion follows thanks to the compactness of the embedding $W^{1,1}(c,d) \hookrightarrow L^1(c,d)$. Indeed, there exist $h_{+}, h_{-} \in L^1(c,d)$ and $\eps_k \to 0^+$ such that
\begin{align*}
    (f_{\eps_k,+})^2 & \to h_{+} \ \text{almost everywhere on $(c,d)$}, \\
    (f_{\eps_k,-})^2 & \to h_{-} \ \text{almost everywhere on $(c,d)$}.
\end{align*}
Thus, $f_{\eps} \to \sqrt{h_{+}} - \sqrt{h_{-}}$ almost everywhere on $(c,d)$. Moreover assumption~(ii) tells us that necessarily $f_0 = \sqrt{h_{+}} - \sqrt{h_{-}}$. 

Now, we turn back to the proof of the \textsc{Claim}. We notice that it is sufficient to show that the function $v_{\eps}$ defined as 
\[
    v_{\eps}(s) =
    \begin{cases}
    \big(f_{\eps}^2\big)'(s) & \text{if $s \in A_{\eps}:=\{ f_{\eps} > 0 \}$}, \\
    0 & \text{if $s \in (c,d)\setminus A_\eps$},
    \end{cases}
\]
is the weak derivative of $(f_{\eps,+})^2$. Let us write $A_{\eps} = \cup_i I_{\eps}^i$ where $\{ I_{\eps}^i \}$ is an at most countable family of disjoint open intervals and let us fix $\phi \in C^{\infty}_c(c,d)$. By assumption $f_{\eps}$ is continuous, in particular $f_{\eps} = 0$ on $\partial I_{\eps}^i$, therefore it holds 
\[
    \int_{c} ^d f_{\eps}^2 \phi' ds
    = \sum_i \int_{I_{\eps}^i} f_{\eps}^2 \phi' ds
    = - \sum_i \int_{I_{\eps}^i} (f_{\eps}^2)' \phi \,ds
    = - \int_{c} ^d (f_{\eps}^2)' \phi \,ds,
\]
from which the conclusion follows.
\end{proof}

\begin{lemma}
    \label{lemma:caratterizzazione}
    Let $W \colon \mathbb{R} \to [0,+\infty)$ be a potential satisfying (W1) and (W2). Suppose that $\psi \colon \mathbb{R} \to \mathbb{R}$ is a continuous function such that:
    \begin{itemize}
        \item[(i)] $\psi \in C^{1}(\Omega)$ where $\Omega := \{ W(\psi) > 0 \}$,
        \item[(ii)] $\dot{\psi}(s)^2 = W(\psi(s))$ for every $s \in \Omega$,
        \item[(iii)] $\psi(0) = \gamma \in (a,b)$.
    \end{itemize}
    Then either $\psi = q_{0} \circ \tau_{\gamma}^{+}$ or $\psi = q_{0} \circ \tau_{\gamma}^{-}$, where $\tau_{\gamma}^{+}$ and $\tau_{\gamma}^{-}$ are defined in Proposition~\ref{prop:blowup}. 
\end{lemma}

\begin{proof}
    The third assumption tells us that $0 \in \Omega$. Let $\mathcal{I} \subseteq \Omega$ be the largest open interval containing $0$. It follows from~(i) and~(ii) that either $\dot{\psi} > 0$ on $\mathcal{I}$ or $\dot{\psi} < 0$ on $\mathcal{I}$. We treat the case $\dot{\psi} > 0$ since the other one is similar. We know that 
    \begin{equation}
        \label{lemma:car:radice}
        \begin{cases}
        \dot{\psi}(s) = \sqrt{W(\psi(s))} & \forall s \in \mathcal{I}, \\ 
        \psi(0) = \gamma \in (a,b). 
        \end{cases}
    \end{equation}
    
    On the other hand $q_0 \circ \tau_{\gamma}^{+}$ solves~(\ref{lemma:car:radice}) on the whole real line (see Remark~\ref{remark:transaltions}), therefore by the uniqueness of the solution and the maximality of $\mathcal{I}$ we have $\mathcal{I} = \mathbb{R}$ and $\psi = q_0 \circ \tau_{\gamma}^{+}$. 
    
    In the case $\dot{\psi} < 0$ we obtain $\psi = q_0 \circ \tau_{\gamma}^{-}$.
\end{proof}

\begin{proof}[Proof of Proposition~\ref{prop:blowup}]
    Before starting with the proof we outline briefly the strategy in four steps. 
    
    \textsc{Step 1}: The bound on $E_{\eps}(u_{\eps}, \mathbb{R}^n)$ implies that there exists $\psi_0 \in W^{1,2}_{\mathsf{loc}}(\mathbb{R})$ such that, up to subsequences, for any real number $m > 0$ it holds
    \begin{equation}
    \label{step1-1}
        \psi_{\eps} \rightharpoonup \psi_0 \ \text{weakly in} \ W^{1,2}(-m,m).
    \end{equation}
    
    \textsc{Step 2}: The bound on $G_{\eps}(u_{\eps}, \mathbb{R}^n)$, together with the previous bound, implies that for any $m > 0$ we have
    \begin{equation}
    \label{step2-1}
        \dot{\psi}_{\eps}^2 - W(\psi_{\eps}) \to 0 \ \text{strongly in} \ W^{1,1}(-m,m).
    \end{equation}
    
    \textsc{Step 3}: Combining the previous steps we notice that for any $m > 0$ the family $\{\dot{\psi}_{\eps}^2\}$ is bounded in $W^{1,1}(-m,m)$. Therefore we are in a position to apply Lemma~\ref{lemma:quadrato} with $f_{\eps} = \dot{\psi}_{\eps}$ (which is a family of continuous functions because $\{ u_\eps \} \subset W^{2,1}_{\mathsf{loc}}(\mathbb{R}^n)$) and $p=2$ to deduce 
    \begin{gather}
    \label{step3-1}
        \dot{\psi}_{\eps} \to \dot{\psi}_0 \ \text{strongly in} \ L^2(-m,m), \\[0.5ex] 
    \label{step3-2}
        \dot{\psi}_0^2 = W(\psi_0) \ \text{a.e. on $(-m,m)$}.
    \end{gather}
    
    \textsc{Step 4}: Using Lemma~\ref{lemma:caratterizzazione} we conclude that if $\psi_0$ is a solution of~(\ref{step3-2}) then necessarily $\psi_0 = q_0 \circ \tau_{\gamma}^{+}$ or $\psi_0 = q_0 \circ \tau_{\gamma}^{-}$. Finally, we improve the convergence in~(\ref{step1-1}) to strong convergence in $W^{2,2}(-m,m)$.
    
    We start with the actual proof. By assumption there exist two constants $\bar{\eps},C_0 > 0$ such that for every $\eps \in (0,\bar{\eps})$ it holds
    \begin{equation}
    \label{stima-funzionali}
        E_{\eps}(u_{\eps}, \mathbb{R}^n) + G_{\eps}(u_\eps, \mathbb{R}^n) \le C_0.
    \end{equation}
 
    By a change of variable in the integral, we can rewrite the functionals in terms of $\psi_{\eps}$:
    \begin{gather} 
        E_{\eps}(u_{\eps}, \mathbb{R}^n) =                                  \int_{-\frac{R_{\eps}}{\eps}}^{+\infty}
        \left( \dot{\psi}_{\eps}^2 + W(\psi_{\eps}) \right)
        \omega_{n-1} (R_{\eps} + \eps s)^{n-1} ds,\label{eq:energy-bu} \\
        G_{\eps}(u_{\eps}, \mathbb{R}^n) = \int_{-\frac{R_{\eps}}{\eps}}^{+\infty}
        \left( \frac{2\ddot{\psi}_{\eps} - W'(\psi_{\eps})}{\eps} + \frac{2(n-1)}{R_{\eps} + \eps s} \dot{\psi}_{\eps} \right)^2
        \left( \dot{\psi}_{\eps}^2 + W(\psi_{\eps}) \right)
        \omega_{n-1} (R_{\eps} + \eps s)^{n-1} ds.\nonumber
    \end{gather}
    Now, for any $m > 0$ there exists $\eps_m \in (0,\bar{\eps})$ such that for any $\eps \in (0,\eps_m)$ it holds
    \begin{equation}
    \label{controllo:raggio}
        R_{\eps} + \eps s \ge \frac{r_0}{2} \quad \forall s \in [-m,m].
    \end{equation}
    Accordingly, for any $\eps \in (0,\eps_m)$ we have
    \begin{align}
        \notag
            C_0 \ge G_{\eps}(u_{\eps}, \mathbb{R}^n) & \ge 
            \int_{-m}^{m}
            \left( \frac{2\ddot{\psi}_{\eps} - W'(\psi_{\eps})}{\eps} + \frac{2(n-1)}{R_{\eps} + \eps s} \dot{\psi}_{\eps} \right)^{2}
            \left( \dot{\psi}_{\eps}^2 + W(\psi_{\eps}) \right)
            \omega_{n-1} \left( \frac{r_0}{2} \right) ^{n-1} ds \\[1.5ex]
        \notag
            & \ge
            \int_{-m}^{m}
            \left( \frac{2\ddot{\psi}_{\eps} - W'(\psi_{\eps})}{\eps} + \frac{2(n-1)}{R_{\eps} + \eps s} \dot{\psi}_{\eps} \right)^2
            \dot{\psi}_{\eps}^2\ 
            \omega_{n-1} \left( \frac{r_0}{2} \right) ^{n-1} ds \\[1.5ex]
        \label{stima1}
            & =
            \int_{-m}^{m}
            \left( \frac{2\ddot{\psi}_{\eps} \dot{\psi}_{\eps} - W'(\psi_{\eps}) \dot{\psi}_{\eps}}{\eps} + \frac{2(n-1)}{R_{\eps} + \eps s} \dot{\psi}_{\eps}^2 \right)^2
            \omega_{n-1} \left( \frac{r_0}{2} \right) ^{n-1} ds.
    \end{align}
    
    \noindent We observe that the term $2\ddot{\psi}_{\eps} \dot{\psi}_{\eps} - W'(\psi_{\eps}) \dot{\psi}_{\eps}$ in the last integral is the derivative of $\dot{\psi}_{\eps}^2 - W(\psi_{\eps})$. In particular, if we set
    \[
    \beta_{\eps} := \dot{\psi}_{\eps}^2
    \quad \text{and} \quad
    g_{\eps} := W(\psi_{\eps})
    \]
    we can rewrite~(\ref{stima1}) as follows:
    \begin{equation}
    \label{stima-star}
        \int_{-m}^{m}
            \left( \frac{ \dot{\beta}_{\eps} - \dot{g}_{\eps} } {\eps} + \frac{2(n-1)}{R_{\eps} + \eps s} \beta_{\eps} \right)^2
            ds 
        \le
        \frac{C_0}{\omega_{n-1}} \left( \frac{2}{r_0} \right)^{n-1}.
    \end{equation}
    Using this notation we deduce from~(\ref{stima-funzionali}) and (\ref{eq:energy-bu}) that for any $\eps \in (0, \eps_m)$ it holds
    \begin{equation}
    \label{stima-starstar}
        \int_{-m}^{m}
        \left( \beta_{\eps} + g_{\eps} \right) ds
        \le
        \frac{C_0}{\omega_{n-1}} \left( \frac{2}{r_0} \right)^{n-1}.
    \end{equation}
    
    Notice also that the constant in the right hand side of~(\ref{stima-star}) and~(\ref{stima-starstar}) does not depend on $m > 0$; this will be important in the sequel.
    
    Exploiting that $W$ satisfies assumption (W3+) we deduce from Remark~\ref{remark:W4} that there exists $\zeta > 0$ such that $W(u) \ge \kappa^2 u^2/2 - \zeta$ for every $u \in \mathbb{R}$. Therefore, the validity of~(\ref{step1-1}) in \textsc{Step 1} immediately follows from~(\ref{stima-starstar}).
    
    Moreover, the fact that $W$ satisfies (W1) implies that $W$ and $W'$ are locally Lipschitz continuous, therefore $\{g_{\eps}\}$ converges locally uniformly to $W(\psi_0)$ and $\{W'(\psi_\eps)\}$ converges locally uniformly to $W'(\psi_0)$. In particular, for any $m > 0$, it follows that
    \begin{equation}
        \label{eq:weak-conv-g}
        g_{\eps} \rightharpoonup W(\psi_0) \ \text{weakly in $W^{1,2}(-m,m)$}.
    \end{equation}
   
   In order to obtain~(\ref{step2-1}) in \textsc{Step 2} we first claim that, up to subsequences, the functions $\beta_{\eps} - g_{\eps}$ converge strongly to a constant in $W^{1,1}(-m,m)$ for any $m > 0$. Then we will show that any such constant must be zero.
   
   To prove the claim it is clearly sufficient to establish, for any $\eps \in (0,\eps_m)$, the following estimates:
   \begin{equation}
   \label{stima-step2}
        \norm*{\beta_{\eps} - g_{\eps}}_{L^1(-m,m)} \le C_1,
        \qquad
        \norm*{\dot{\beta}_{\eps} - \dot{g}_{\eps}}_{L^1(-m,m)} \le C_1 (\sqrt{m} + 1) \eps,
   \end{equation}
   for some constant $C_1 > 0$ which does not depend on $m > 0$. Notice that from~(\ref{stima-starstar}) we have
   \begin{equation}
        \label{eq:discrepancy}
        \norm*{\beta_{\eps} - g_{\eps}}_{L^1(-m,m)}
        \le
        \norm*{\beta_{\eps} + g_{\eps}}_{L^1(-m,m)}
        \le
        \frac{C_0}{\omega_{n-1}} \left( \frac{2}{r_0} \right)^{n-1}.
   \end{equation}
   Moreover, using~(\ref{controllo:raggio}), (\ref{stima-star}) and~(\ref{stima-starstar}) we have
   \begin{align}
        \notag
        \norm*{\dot{\beta}_{\eps} - \dot{g}_{\eps}}_{L^1}
        & \le
        \eps \norm*{\frac{2(n-1)}{R_{\eps} + \eps s} \beta_{\eps}}_{L^1}
        +
        \norm*{\dot{\beta}_{\eps} - \dot{g}_{\eps} + \frac{2\eps(n-1)}{R_{\eps} + \eps s} \beta_{\eps}}_{L^1} \\[1.5ex]
        \notag
        & \le
        \eps \norm*{\frac{2(n-1)}{R_{\eps} + \eps s}}_{L^{\infty}} \cdot \norm*{\beta_{\eps}}_{L^1}
        +
        \sqrt{2m} \cdot \norm*{\dot{\beta}_{\eps} - \dot{g}_{\eps} + \frac{2\eps(n-1)}{R_{\eps} + \eps s} \beta_{\eps}}_{L^2} \\[1.5ex]
        \label{eq:ddiscrepancy}
        & \le
        \eps \left\{ \frac{2(n-1)C_0}{\omega_{n-1}} \left( \frac{2}{r_0} \right)^n + \sqrt{ \frac{2mC_0}{\omega_{n-1}} \left( \frac{2}{r_0} \right)^{n-1} }  \right\},
   \end{align} 
   where all norms are implicitly computed on $(-m,m)$. The validity of~(\ref{stima-step2}) is now an easy consequence of~(\ref{eq:discrepancy}) and~(\ref{eq:ddiscrepancy}).
   
   We pick a subsequence $\{\eps_k\}$ such that $\{\beta_{\eps_k} - g_{\eps_k}\}$ converges strongly to a constant $c_0$ in $W^{1,1}(-m,m)$ as $k\to +\infty$. From~(\ref{stima-step2}) it follows that
   \begin{equation}
   \label{stima-c0}
        2m \cdot c_0 = \lim_{k \to +\infty} \norm*{\beta_{\eps_k} - g_{\eps_k}}_{L^1(-m,m)} \le C_1
   \end{equation}
   and recalling that $C_1$ in~(\ref{stima-c0}) does not depend on $m$ we deduce $c_0=0$, which completes the proof of~\textsc{Step 2}.
   
   \smallskip
   
   So far we know that $\{\beta_{\eps} - g_{\eps}\}$ converges strongly to zero in $W^{1,1}(-m,m)$, thus if we are able to prove that $\{\beta_{\eps} - g_{\eps}\}$ also converges to $\dot{\psi}_0^2 - W(\psi_0)$ then necessarily $\dot{\psi}_0^2 = W(\psi_0)$ almost everywhere on $(-m,m)$ and hence, since $m>0$ is arbitrary, also in $\R$. 
   We start by proving~(\ref{step3-1}) in \textsc{Step 3}. This is a direct consequence of Lemma~\ref{lemma:quadrato} applied with $f_{\eps} = \dot{\psi}_{\eps}$, provided that we check assumptions \ref{hyp:lemma-quadrato-1} and \ref{hyp:lemma-quadrato-2}. The validity of assumption \ref{hyp:lemma-quadrato-1} is obtained from~(\ref{stima-step2}) and a triangular inequality, because~(\ref{eq:weak-conv-g}) implies that $\{g_\eps\}$ is bounded in $W^{1,1}(-m,m)$. Assumption~\ref{hyp:lemma-quadrato-2} follows from~(\ref{step1-1}).
   
   Using this fact we can show that $\{g_{\eps}\}$ converges strongly to $W(\psi_0)$ in $W^{1,2}(-m,m)$. Indeed from~(\ref{eq:weak-conv-g}) we already know that weak convergence holds; on the other hand, from~(\ref{step3-1}) we immediately conclude that
   \begin{equation}
   \label{conv:g}
      \dot{g}_{\eps} = W'(\psi_{\eps}) \dot{\psi}_{\eps} \to W'(\psi_0) \dot{\psi}_0 \ \text{strongly in} \ L^{2}(-m,m).
   \end{equation}
   
   The convergence of $\{\beta_{\eps}\}$ to $\dot{\psi}_0^2$ is also a consequence of~(\ref{step3-1}); indeed we know from~(\ref{step2-1}), (\ref{eq:weak-conv-g}) and~(\ref{conv:g}) that $\{\beta_{\eps}\}$ converges strongly to $W(\psi_0)$ in $W^{1,1}(-m,m)$. On the other hand~(\ref{step3-1}) ensures the existence of a subsequence $\{\beta_{\eps_k}\}$ converging almost everywhere to $\dot{\psi}_0^2$ and therefore~(\ref{step3-2}) in \textsc{Step 3} holds. 
   
   \smallskip
   
   At this point we want to conclude that either $\psi_0 = q_0 \circ \tau_{\gamma}^{+}$ or $\psi_0 = q_0 \circ \tau_{\gamma}^{-}$ using Lemma~\ref{lemma:caratterizzazione} with $\psi = \psi_0$. We have to check that the assumptions are satisfied. Notice that $\psi_{\eps}(0) = u_{\eps}(R_{\eps})$ and by assumption $u_{\eps}(R_{\eps}) \to \gamma \in (a,b)$ as $\eps \to 0^+$, hence assumption (iii) holds. Moreover, the validity of~(ii) follows from~(\ref{step3-2}) once we show that~(i) holds. Therefore, we focus on proving that $\psi_0 \in C^1(\Omega)$, where $\Omega = \{ W(\psi_0) > 0 \}$. Given $\delta>0$, we consider the open set
   \[
    \Omega_{m, \delta} := \{ s \in (-m,m) : W(\psi_0(s)) > \delta \}.
   \]
   It is enough to show that $\psi_0 \in C^1(\Omega_{m,\delta})$ for any $m > 0$ and $\delta > 0$. It is clear from the local uniform convergence of $\{g_\eps\}$ to $W(\psi_0)$ that there exists $\eps_{m,\delta} \in (0,\eps_m)$ such that for any $\eps \in (0,\eps_{m,\delta})$ it holds $W(\psi_{\eps}(s)) \ge \delta/2$ for every $s \in \Omega_{m,\delta}$. Starting from the first line in~(\ref{stima1}) we obtain the following estimate
   \begin{equation}
        \label{step4:verifica}
        \int_{\Omega_{m,\delta}} \left( 2\ddot{\psi}_{\eps} - W'(\psi_{\eps}) + \frac{2\eps(n-1)}{R_{\eps} + \eps s} \dot{\psi}_{\eps} \right)^2
        ds 
        \le
        \frac{\eps^2}{\delta} \cdot \frac{2C_0}{\omega_{n-1}} \left( \frac{2}{r_0} \right)^{n-1},
   \end{equation}
   which is valid for any $\eps \in (0,\eps_{m,\delta})$. Taking into account~(\ref{step4:verifica}) and the fact that
   \begin{equation}
        \label{step4:sup}
        \sup_{\eps\in (0,\eps_m)} \left\{ \norm{W'(\psi_{\eps})}_{L^2(-m,m)} + \norm{\dot{\psi}_{\eps}}_{L^2(-m,m)} \right\} < +\infty,
   \end{equation}
   we deduce that $\{ \ddot{\psi}_{\eps} \}_{\eps \in (0,\eps_{m,\delta})}$ is bounded in $L^2(\Omega_{m,\delta})$ which implies $\psi_0 \in W^{2,2}(\Omega_{m,\delta})$ and in particular $\psi_0 \in C^1(\Omega_{m,\delta})$.
   Finally, thanks to Lemma~\ref{lemma:caratterizzazione} we obtain that either $\psi_0 = q_0 \circ \tau_{\gamma}^{+}$ or $\psi_0 = q_0 \circ \tau_{\gamma}^{-}$.
   
   Now, we prove that $\{\psi_{\eps}\}$ converges strongly to $\psi_0$ in $W^{2,2}(-m,m)$. Combining the precise expression of $\psi_0$ that we have just obtained with~(i) in Lemma~\ref{lemma:q0} and Remark~\ref{remark:transaltions} we easily deduce that $a < \psi_0(s) < b$ for every $s \in \mathbb{R}$. 
   Therefore, for any $m > 0$ there exists $\delta_m > 0$ such that $\Omega_{m,\delta} = (-m,m)$ for any $\delta \in (0,\delta_m)$.
   
   Using the expression of $\Omega_{m,\delta}$ together with~(\ref{step4:verifica}) and~(\ref{step4:sup}) we get
   \[
        \norm*{ 2\ddot{\psi}_{\eps} - W'(\psi_{\eps})}_{L^2(-m,m)} \le C(m) \cdot \eps,
   \]
   where $C(m) > 0$ is a constant that depends on $m > 0$. In particular, by~(\ref{eq:weak-conv-g}), this means that
   \[
        2\ddot{\psi}_{\eps} \to W'(\psi_0) = 2 \ddot{\psi}_0 \ \text{strongly in} \ L^2(-m,m).
   \]
   
   At this point the proof of the first part of Proposition~\ref{prop:blowup} is almost complete; to conclude we pick a sequence $\eps_k \to 0^+$ and, by a standard diagonal argument, we can find a (not relabelled) subsequence and an increasing sequence $m_k\to +\infty$ such that $\eps_k m_k\to 0$ and (\ref{eq:H2conv}) holds.
   
   Since $\tau_{\gamma}^+$ and $\tau_{\gamma}^-$ are isometries of $\mathbb{R}$ it is equivalent to prove~(\ref{eq:state-limit}) with $\psi_0$ instead of $q_0$.
   First of all we notice that the convergence of the $L^2(-m_k,m_k)$-norm of $\dot{\psi}_{\eps_k}$ to the $L^2(\mathbb{R})$-norm of $\dot{\psi}_0$ as $k \to +\infty$ is an immediate consequence of~(\ref{eq:H2conv}). Therefore, it is enough to show that
   \[
        \lim_{k \to +\infty} \int_{-m_k}^{m_k} \abs*{W(\psi_{\eps_k}) - W(\psi_0)} ds = 0.
   \]
   Now, for any $s \in (-m_k,m_k)$ there exists $\zeta_k(s)$ between $\psi_{\eps_k}(s)$ and $\psi_0(s)$ such that
   \begin{align}
    \notag
        \abs*{W(\psi_{\eps_k}(s)) - W(\psi_0(s))}
        & \le
        \abs*{W'(\psi_0(s))} \cdot \abs*{\psi_{\eps_k}(s) - \psi_0(s)} \\[1ex]
        \label{eq:proof-state-limit}
        & +
        \frac{\abs*{W''\big(\zeta_k(s)\big)}}{2} \cdot \big(\psi_{\eps_k}(s) - \psi_0(s)\big)^2.
   \end{align}
    Moreover, from~(\ref{eq:H2conv}) and Remark~\ref{remark:H2-C1} we know that if $k$ is large enough then
    \[
    \abs*{\psi_{\eps_k}(s)} \le \abs*{\psi_0(s)} + 1 \le \max \{ \abs{a}, \abs{b} \} + 1 \quad \forall s \in (-m_k,m_k).
    \]
    In particular, there exists a constant $M > 0$ such that $\abs{W''(\zeta_k(s))} \le M$ for every $s \in (-m_k,m_k)$ and $k$ large enough. Combining~(\ref{eq:proof-state-limit}) with Hölder inequality we get
    \begin{align}
        \notag
        \int_{-m_k}^{m_k} \abs*{W(\psi_{\eps_k}) - W(\psi_0)} ds
        & \le 
        \bigg( \int_{-m_k}^{m_k} \big( W'(\psi_0) \big)^2 ds \bigg)^{1/2} \bigg( \int_{-m_k}^{m_k} (\psi_{\eps_k} - \psi_0 )^2 ds \bigg)^{1/2} \\[1ex]
        \label{eq:taylor-int}
        & + 
        \frac{M}{2} \int_{-m_k}^{m_k} (\psi_{\eps_k} - \psi_0)^2 ds.
    \end{align}
    If we prove that $W'(\psi_0) \in L^2(\mathbb{R})$ then the conclusion follows from~(\ref{eq:H2conv}) and~(\ref{eq:taylor-int}). By a change of variable in the integral we observe that
    \[
    \int_{-\infty}^{+\infty} \big( W'(\psi_0(s)) \big)^2 ds
    =
    \int_{-\infty}^{+\infty} \frac{\big( W'(\psi_0(s)) \big)^2}{\sqrt{W(\psi_0(s))}} \abs{\dot{\psi}_0(s)} ds
    =
    \int_a^b \frac{( W'(\psi) )^2}{\sqrt{W(\psi)}} d\psi < +\infty,
    \]
    where the last integral is finite because $W > 0$ in $(a,b)$, $W \in C^2(\mathbb{R})$ and moreover
    \[
        \lim_{u \to a^+} \frac{W'(u)}{\sqrt{W(u)}} = \sqrt{2W''(a)}, \quad \lim_{u \to b^-} \frac{W'(u)}{\sqrt{W(u)}} = - \sqrt{2W''(b)}.
    \]
    This concludes the proof of Proposition~\ref{prop:blowup}.
\end{proof}

The following lemma roughly says that the second order quantity $G_{\eps}(u_{\eps}, \Omega_{\eps})$ controls the first order quantity $E_{\eps}(u_{\eps}, \Omega_{\eps})$ in regions $\Omega_{\eps}$ that are not too large in measure and in which $u_{\eps}$ stays close to the zeros $\{a,b\}$ of the potential $W$. We prove it without the radiality assumption, since it does not simplify the argument in this case.

\begin{lemma}
    \label{lemma:out-of-zeros} 
    Let $W \colon \mathbb{R} \to [0,+\infty)$ be a potential satisfying assumptions (W1), (W2) and (W3+).
    Let $\{ \Omega_{\eps} \}$ be a family of bounded open subsets of $\mathbb{R}^n$ with smooth boundary and let $\{ u_{\eps} \} \subset W_{\mathsf{loc}}^{2,1}(\mathbb{R}^n)\cap W^{1,2}_{\mathsf{loc}}(\R^n)$ be a family of functions such that $u_\eps \in L^\infty(\Omega_\eps)$ and
    \[
        \limsup_{\eps\to 0^+} \Big( \Leb^n(\Omega_{\eps}) + G_{\eps}(u_\eps, \Omega_{\eps}) \Big) < +\infty.
    \]
    Let us suppose that the following assumptions hold (at least for a sequence $\eps_k \to 0^+$):
    \begin{itemize}
        \item[(1)]  we have
                    \[
                    \lim_{\eps \to 0^+} \int_{\partial \Omega_{\eps}} \eps \sqrt{W(u_{\eps})} \cdot \abs*{ \frac{\partial u_{\eps}}{\partial \nu}} \ d \mathcal{H}^{n-1} = 0,
                    \]
                    where $\nu$ denotes the outward unit normal to $\partial \Omega_\eps$,
        \item[(2)] for any $\delta > 0$ there exists $\eps_{\delta} > 0$ such that for any $ \eps \in (0,\eps_{\delta})$ it holds either
                    \begin{gather*}
                        u_{\eps}  \le a + \delta \ \text{almost everywhere in $\Omega_{\eps}$} \\
                        \text{or} \\
                        u_{\eps}  \ge b - \delta \ \text{almost everywhere in $\Omega_{\eps}$}.
                    \end{gather*}
    \end{itemize}
    Then (at least on the sequence $\{\eps_k\}$)
    \[
        \lim_{\eps \to 0^+} E_{\eps}(u_{\eps}, \Omega_{\eps}) = 0.
    \]
\end{lemma}

\begin{proof}
    Let $g \in C^1(\mathbb{R})$ and notice that
    \[
    \int_{\Omega_{\eps}} g(u_{\eps}) \Delta u_{\eps} \,dx
    =
    \int_{\partial \Omega_{\eps}} g(u_{\eps}) \frac{\partial u_{\eps}}{\partial \nu} \,d\mathcal{H}^{n-1}
    - 
    \int_{\Omega_{\eps}} g'(u_{\eps}) \abs*{\nabla u_{\eps}}^2 \,dx.
    \]
    Therefore,
    \begin{multline}
        \label{eq:lemma-ooz-1}
        \int_{\Omega_{\eps}} g'(u_{\eps}) \left( \eps \abs*{\nabla u_{\eps}}^2 + \frac{W(u_{\eps})}{\eps} \right)\,dx\\
        =
        \int_{\partial \Omega_{\eps}} \eps g(u_{\eps}) \frac{\partial u_{\eps}}{\partial \nu}\,d\mathcal{H}^{n-1} - \int_{\Omega_{\eps}} \left[g(u_{\eps}) \eps \Delta u_{\eps} - g'(u_{\eps}) \frac{W(u_{\eps})}{\eps}\right]dx.
    \end{multline}

    As we already know from Remark~\ref{remark:W4}, if $W$ satisfies both (W2) and (W3+) then its only zeros are $u=a$ and $u=b$. Therefore, the function
    \[
    g(u) :=
        \begin{cases}
        2\sqrt{W(u)} & \text{if $u \in (-\infty,a]$}, \\
        -2\sqrt{W(u)} & \text{if $u \in [a,b]$}, \\
        2\sqrt{W(u)} & \text{if $u \in [b,+\infty)$},
        \end{cases}
    \]
    is of class $C^1(\mathbb{R})$ as a consequence of assumption (W1) (to see this it is enough to compute $g'(u)$ and its limits as $u\to a$ and $u\to b$). Moreover, it is easy to check that $g$ satisfies the relation $g' \cdot 2W = W' \cdot g$. Thus, we can rewrite~(\ref{eq:lemma-ooz-1}) as follows:
    \begin{multline}
        \label{eq:int-by-parts}
        \int_{\Omega_{\eps}} g'(u_{\eps}) \left( \eps \abs*{\nabla u_{\eps}}^2 + \frac{W(u_{\eps})}{\eps} \right)dx \\
        =
        \int_{\partial \Omega_{\eps}} \eps g(u_{\eps}) \frac{\partial u_{\eps}}{\partial \nu} \,d\mathcal{H}^{n-1} - 
        \int_{\Omega_{\eps}} g(u_{\eps}) \left( \eps \Delta u_{\eps} - \frac{W'(u_{\eps})}{2\eps} \right)dx.
    \end{multline}
    We observe that $\abs*{g'(u)} = \abs*{W'(u)}/\sqrt{W(u)}$ for every $u \notin \{a,b\}$, therefore from~(\ref{eq:W4-revisited}) and the continuity of $g'$ we deduce that there exists $\delta_0 > 0$ such that
    \begin{equation}
        \label{eq:constant-sign}
        \abs{g'(u)} \ge \kappa/2 \quad \text{for every $u \notin (a+\delta_0, b-\delta_0)$}. 
    \end{equation}

    We pick $\eps_{\delta_0} > 0$ such that assumption~(2) is satisfied; thus for any $\eps \in (0,\eps_{\delta_0})$ we have
    \begin{align}
        \notag
        \int_{\Omega_{\eps}} \left(\eps \abs*{\nabla u_{\eps}}^2 + \frac{W(u_{\eps})}{\eps}\right) dx
        & \le
        \frac{2}{\kappa} \int_{\Omega_{\eps}} \abs{g'(u_{\eps})} \left( \eps \abs*{\nabla u_{\eps}}^2 + \frac{W(u_{\eps})}{\eps} \right) dx\\[1ex]
        \notag
        & = \frac{2}{\kappa} \abs*{ \int_{\Omega_{\eps}} g'(u_{\eps}) \left( \eps \abs*{\nabla u_{\eps}}^2 + \frac{W(u_{\eps})}{\eps} \right) dx } \\[1ex]
        \notag
        & \le 
        \frac{2}{\kappa} \left\{ \int_{\partial \Omega_{\eps}} \eps \abs*{g(u_{\eps})} \abs*{\frac{\partial u_{\eps}}{\partial \nu}}\,d\mathcal{H}^{n-1} \right.\\[0.5ex]
        \notag
        &\quad\quad
        \left.+\int_{\Omega_{\eps}} \frac{\abs*{g(u_{\eps})}}{2} \cdot \abs*{ 2 \eps \Delta u_{\eps} - \frac{W'(u_{\eps})}{\eps} }\,dx \right\} \\[1ex]
        \label{eq:ooz1}
        & \le
        \frac{2}{\kappa} \left\{ \int_{\partial \Omega_{\eps}} \eps \abs*{g(u_{\eps})} \abs*{\frac{\partial u_{\eps}}{\partial \nu}}\,d\mathcal{H}^{n-1} + \sqrt{\eps \Leb^{n}(\Omega_{\eps}) \cdot G_{\eps}(u_{\eps}, \Omega_{\eps})} \right\},
    \end{align}
    where to pass from the first to the second line we used~(\ref{eq:constant-sign}) and the continuity of $g'$, to pass from the second to the third line we used~(\ref{eq:int-by-parts}), and to pass from the third to the fourth line we used Jensen inequality together with the inequality
    \[
        \frac{g(u_{\eps})^2}{4 \eps} \le \eps \abs*{\nabla u_{\eps}}^2 + \frac{W(u_{\eps})}{\eps}.
    \]
    
    Now, the right-hand side of~(\ref{eq:ooz1}) goes to zero as $\eps \to 0^+$, since the first addendum goes to zero thanks to assumption~(1), while the second one goes to zero because $\Leb^{n}(\Omega_{\eps})$ and $G_{\eps}(u_{\eps}, \Omega_{\eps})$ are bounded as $\eps \to 0^+$.
\end{proof}

\begin{remark}
    \label{remark:bdry-radial}
    \rm Under the radiality assumption, if $\Omega_{\eps} = B_{d_{\eps}}$ for some $d_{\eps} > 0$ then the integral in assumption (1) of Lemma~\ref{lemma:out-of-zeros} reduces to
    \[
    \int_{\partial B_{d_{\eps}}} \eps \sqrt{W(u_{\eps})} \cdot \abs*{\dot{u}_{\eps}} \ d\mathcal{H}^{n-1} 
    =
    \omega_{n-1} d_{\eps}^{n-1} \eps \sqrt{W(u_{\eps}(d_{\eps}))} \cdot \abs*{\dot{u}_{\eps}(d_{\eps})}.
    \]
\end{remark}

Finally, we can prove Theorem~\ref{theorem:main2}. In the proof we use Proposition~\ref{prop:blowup} with $\gamma_0:=(a+b)/2$ several times. In this case $\tau_{\gamma_0}^+(s)=s$ and $\tau_{\gamma_0}^-(s)=-s$. In order to simplify the notation we set $\widetilde{q}_0(s) := q_0(-s)$ and from now on we use $\widetilde{q}_0$ in place of $q_0 \circ \tau_{\gamma_0}^-$.

\begin{proof}[Proof of Theorem~\ref{theorem:main2}]
    We divide the proof in seven steps.
    
    \textsc{Step 1:} Let us fix $r_0 > 0$ and let us define $Z_{\eps}(r_0):=\{ r \in [r_0,+\infty) : u_{\eps}(r) = \gamma_0 \}$. We claim that there exists a constant $R > 0$ that does not depend on $r_0$ such that, if $\eps$ is small enough, then $Z_{\eps}(r_0)$ is a discrete subset of $[r_0, R]$. 
    
    The uniform boundedness of $Z_{\eps}(r_0)$ is a direct consequence of Proposition~\ref{prop:blowup}. Indeed let us suppose by contradiction that there exist a sequence $\eps_k \to 0^+$ and a family of points $z_k \to +\infty$ such that $u_{\eps_{k}}(z_k) = \gamma_0$ for every $k \in \mathbb{N}$. Then, by Proposition~\ref{prop:blowup}, there exists a sequence $\{ m_k \}$ of positive real numbers such that $m_k \to +\infty$, $m_k \eps_k \to 0$ as $k \to +\infty$ and such that the blow-ups $\psi_{\eps_k}$ of $u_{\eps_k}$ at $z_{\eps_k}$ satisfy
    \begin{equation*}
            \norm*{\psi_{\eps_k} - q_0}_{W^{2,2}(-m_k,m_k)} \to 0 \quad \text{or} \quad \norm*{\psi_{\eps_k} - \widetilde{q}_0}_{W^{2,2}(-m_k,m_k)} \to 0.
    \end{equation*}
    
    Therefore, for any fixed $M > 0$ we have $z_{\eps_k} \ge M$ for $k$ large enough and also
    \begin{align*}
        M^{n-1} \sigma_W^{a,b}
        & =
        M^{n-1} \int_{-\infty}^{+\infty} \left( \dot{q}_0^2 + W(q_0) \right) ds \\[1ex]
        & =
        \lim_{k \to +\infty} \int_{-m_k}^{m_k}
        \left( \dot{\psi}_{\eps_k}^2 + W(\psi_{\eps_k}) \right)
        (M + \eps_k s)^{n-1} ds \\[1ex]
        & \le 
        \liminf_{k \to +\infty} \int_{-m_k}^{m_k}
        \left( \dot{\psi}_{\eps_k}^2 + W(\psi_{\eps_k}) \right)
        (z_{\eps_k} + \eps_k s)^{n-1} ds \\[1ex]
        & =
        \liminf_{k \to +\infty} \int_{z_k - m_k \eps_k}^{z_k + m_k \eps_k} \left( \eps_k \dot{u}_{\eps_k}^2 + \frac{W(u_{\eps_k})}{\eps_k} \right) r^{n-1} dr \\[1ex]
        & \le 
        \frac{E_{\eps_k}(u_{\eps_k}, \mathbb{R}^n)}{\omega_{n-1}},
    \end{align*}
    where in the second line we used~(\ref{eq:state-limit}). Thus we reach a contradiction if $M$ is sufficiently large.
    
    As for the discreteness, let us assume by contradiction that there exist a sequence $\eps_k \to 0^+$ and a family $\{ z_k \}$ of points such that for every $k$ we have $z_k \in Z_{\eps_k}(r_0)$ and $z_k$ is not isolated in $Z_{\eps_k}(r_0)$. Then, by Proposition~\ref{prop:blowup}, there exists a (not relabelled) subsequence such that the blow-ups $\psi_{\eps_k}$ of $u_{\eps_k}$ at $z_k$ converge strongly to $\psi_0$ in $W^{2,2}(-1,1)$ where either $\psi_0 = q_0$ or $\psi_0 = \widetilde{q}_0$. Since $W^{2,2}$ convergence implies $C^1$ convergence and $q_0$ has nonvanishing derivative at $s=0$, this implies that $s=0$ is isolated in $\{ s \in (-1,1) : \psi_{\eps_k}(s) = \gamma_0 \}$ if $k$ is large enough. Then, we deduce that also $z_k$ is isolated in $Z_{\eps_k}(r_0)$, a contradiction.
    
    Therefore, if $\eps$ is small enough, we can order the points in $Z_{\eps}(r_0)$ in decreasing order, namely we can set $Z_{\eps}(r_0) = \{ z_{\eps}^1, z_{\eps}^2, \dots \}$, with $z_{\eps}^1 > z_{\eps}^2 > \cdots$.
    
    \smallskip
    
    \textsc{Step 2:} For every fixed $r_0 > 0$ the cardinality of $Z_{\eps}(r_0)$ is uniformly bounded if $\eps$ is sufficiently small.
    
    In fact, if this is not the case, by a diagonal argument we can find a sequence $\eps_k \to 0^+$ and a sequence $\{ z_i \} \subset [r_0, R]$ of points such that for every $i \in \mathbb{N}$ we have $z_{\eps_k}^i \to z_i$ as $k \to +\infty$.
    
    Then, by Proposition~\ref{prop:blowup}, for every $i \in \mathbb{N}$ there exists a sequence $\{ m_k^i \}$ of positive real numbers such that $m_k^i \to +\infty$, $m_k^i \eps_k \to 0$ as $k \to +\infty$ and such that the blow-ups $\psi_{\eps_k}^i$ of $u_{\eps_k}$ at $z_{\eps_k}^i$ satisfy
    \begin{equation}
        \label{eq:conv-bu}
            \norm*{\psi_{\eps_k}^i - q_0}_{W^{2,2}(-m_k^i,m_k^i)} \to 0 \quad \text{or} \quad \norm*{\psi_{\eps_k}^i - \widetilde{q}_0}_{W^{2,2}(-m_k^i,m_k^i)} \to 0.
    \end{equation}
    
    We claim that for every $N \in \mathbb{N}$ there exists a positive integer $k_N$ such that for every $k \ge k_N$ it holds
    \begin{equation}
        \label{eq:more-than-hausdorff}
        z_{\eps_k}^j \notin (z_{\eps_k}^i - m_k^i \eps_k, z_{\eps_k}^i + m_k^i \eps_k) \ \text{for every $i,j \le N$ with $i \neq j$}.
    \end{equation}
    Suppose by contradiction that~(\ref{eq:more-than-hausdorff}) does not hold; then there exist $N \in \mathbb{N}$ and $i, j \le N$ with $i \neq j$ such that the set
    \[
        Q_{ij} := \{ k \in \mathbb{N} : z_{\eps_k}^j \in (z_{\eps_k}^i - m_k^i \eps_k, z_{\eps_k}^i + m_k^i \eps_k) \}
    \]
    is infinite, in particular there exists a sequence $\{ k_h \} \subset Q_{ij}$ such that $k_h \to +\infty$ as $h \to +\infty$.
    Now, we can write $z_{\eps_{k_h}}^j = z_{\eps_{k_h}}^i + s_h \eps_{k_h}$ for some $s_h \in (-m_{k_h}^i,m_{k_h}^i) \setminus \{ 0 \}$.
    
    Without loss of generality, in~(\ref{eq:conv-bu}) we can assume that $\psi_{\eps_k}^i$ converges to $q_0$, hence from Remark~\ref{remark:H2-C1} we deduce that
    \begin{equation*}
        \limsup_{h \to +\infty} \, \abs*{\gamma_0 - q_0( s_h )}
        =
        \limsup_{h \to +\infty} \, \abs*{\psi_{\eps_{k_h}}^i(s_h) -  q_0( s_h )}
        \le 
        \limsup_{k \to +\infty} \, \norm*{\psi_{\eps_k}^i - q_0}_{L^{\infty}(-m_k^i,m_k^i)} = 0.
    \end{equation*}
    The previous estimate, together with (ii) in Lemma~\ref{lemma:q0}, implies that $s_h \to 0$ as $h \to +\infty$.
    
    Moreover, by the Mean Value Theorem we deduce that for every $h \in \mathbb{N}$ there exists $\zeta_h \in \mathbb{R}$ such that $\abs{\zeta_h} \le \abs{s_h}$ and with the property that $\dot{\psi}_{\eps_{k_h}}^i(\zeta_h) = 0$. This leads to a contradiction because from~(\ref{eq:conv-bu}) and Remark~\ref{remark:H2-C1} we know that $\{\psi_{\eps_k}\}$ converges to $q_0$ in $C^1$ but on the other hand $\zeta_h \to 0$ and $\dot{q}_0(0) \neq 0$.
    
    Therefore (\ref{eq:more-than-hausdorff}) holds. Dividing all the $m_k^i$ by 2, we can assume in addition that the sets $(z_{\eps_k}^i - m_k^i \eps_k, z_{\eps_k}^i + m_k^i \eps_k)$ corresponding to different values of $i \le N$ are disjoint if $k \ge k_N$. Moreover, we still have that $m_k^i \to +\infty$, $m_k^i \eps_k^i \to 0$ as $k \to +\infty$ and that~(\ref{eq:conv-bu}) holds; in particular we deduce
    \begin{align*}
        N r_0^{n-1} \sigma_W^{a,b}
        & \le
        \sum_{i=1}^N z_i^{n-1} \int_{-\infty}^{+\infty} \left( \dot{q}_0^2 + W(q_0) \right) ds \\[1ex]
        & =
        \lim_{k \to +\infty} \sum_{i=1}^N \int_{-m_k^i}^{m_k^i}
        \left(  \left( \dot{\psi}_{\eps_k}^i \right)^2 + W(\psi_{\eps_k}^i) \right)
        (z_{\eps_k}^i + \eps_k s)^{n-1} ds \\[1ex]
        & =
        \lim_{k \to +\infty} \sum_{i=1}^N \int_{z_{\eps_k}^i - m_k^i \eps_k}^{z_{\eps_k}^i + m_k^i \eps_k} \left( \eps_k \dot{u}_{\eps_k}^2 + \frac{W(u_{\eps_k})}{\eps_k} \right) r^{n-1} dr \\[1ex]
        & \le
        \limsup_{k \to +\infty} \int_{0}^{+\infty} \left( \eps_k \dot{u}_{\eps_k}^2 + \frac{W(u_{\eps_k})}{\eps_k} \right) r^{n-1} dr \\[1ex]
        & \le 
        \limsup_{k \to +\infty} \frac{E_{\eps_k}(u_{\eps_k}, \mathbb{R}^n)}{\omega_{n-1}},
    \end{align*}
    where in the second line we used~(\ref{eq:state-limit}); as before we reach a contradiction if $N$ is sufficiently large.
    
    \smallskip
    
    \textsc{Step 3:} Now, we build the family of radii $( r_i )_{i \in I}$ that verifies (\ref{eq:somma_ri_finita}) and such that (\ref{eq:int-var}) holds, possibly after extracting a subsequence.
    
    Combining \textsc{Step 2} with a diagonal argument it is possible to find a sequence $\eps_k \to 0^+$ such that for every $r > 0$ the cardinality of $Z_{\eps_k}(r)$ is eventually equal to some $N_r \in \mathbb{N}$. Namely for every $r > 0$ there exists a positive integer $k_r$ such that for every $k \ge k_r$ we have $Z_{\eps_k}(r) = \{ z_{\eps_k}^1, \dots, z_{\eps_k}^{N_r} \}$, with $z_{\eps_k}^1 > z_{\eps_k}^2 > \cdots > z_{\eps_k}^{N_r}$. 
    
    Let us set
    \[
        N_0:=\sup_{r>0} N_r
        \quad \text{and} \quad
        I = \begin{cases}
        \N &\text{if $N_0=+\infty$},\\
        \{1,\dots,N_0\} &\text{otherwise}.\end{cases}
    \]
    
    Then, possibly extracting another subsequence that we do not relabel, we can assume that for every $i \in I$ there exists $r_i > 0$ such that $z_{\eps_k}^i \to r_i$ as $k \to +\infty$. Clearly $r_i \ge r_{i+1}$ for every possible $i$.
    
    \smallskip
    
    \textsc{Step 4:} By Proposition~\ref{prop:blowup}, for every $i \in I$ we can find $\{ m_k^i \}$ such that $m_k^i \to +\infty$, $m_k^i \eps_k \to 0$ as $k \to +\infty$ and such that the blow-ups $\psi_{\eps_k}^i$ of $u_{\eps_k}$ at $z_{\eps_k}^i$ satisfy~(\ref{eq:conv-bu}) as $k \to +\infty$. 
    
    As in the proof of \textsc{Step 2}, we can choose $\{ m_k^i \}$ such that for every $r > 0$ there exists a positive integer $k_r$ such that for every $k \ge k_r$ the sets $(z_{\eps_k}^i - m_k^i \eps_k, z_{\eps_k}^i + m_k^i \eps_k)$ corresponding to different values of $i \le N_r$ are disjoint. Let us introduce
    \[
        c_k^i := z_{\eps_k}^i - m_k^i \eps_k
        \quad \text{and} \quad
        d_k^i := z_{\eps_k}^i + m_k^i \eps_k.  
    \]
    
    For every $0 < c < d$ we set $A_{c,d} := B_d \setminus \overline{B}_c$ . Now, we prove that for every integer $i \in I$
    \begin{equation*}
        \mu_{\eps_k} \mres A_{c_k^i, d_k^i} \overset{*}{\rightharpoonup} \mathcal{H}^{n-1} \mres \partial B_{r_i}.
    \end{equation*}
    Since $\mu_{\eps}$ and $u_{\eps}$ are radially symmetric, this is equivalent to show that
    \begin{equation}
        \label{eq:conv-near-zeros-red}
        \left( \eps_k \dot{u}_{\eps_k}(r)^2 + \frac{W(u_{\eps_k}(r))}{\eps_k} \right) r^{n-1} \Leb^1 \mres (c_k^i, d_k^i) \overset{*}{\rightharpoonup} \sigma_W^{a,b} r_i^{n-1} \delta_{r_i}.
    \end{equation}
    Let $\phi \in C_0(0,+\infty)$; using~(\ref{eq:state-limit}) we obtain
    \begin{align}
        & \lim_{k \to +\infty} \int_{c_k^i}^{d_k^i} 
        \left( \eps_k \dot{u}_{\eps_k}^2  + \frac{W(u_{\eps_k})}{\eps_k} \right) \phi(r) r^{n-1} dr \nonumber\\
        =\ & \lim_{k \to +\infty} \int_{-m_k^i}^{m_k^i}
     \left( \left( \dot{\psi}_{\eps_k}^i \right)^2 + W(\psi_{\eps_k}^i) \right) \phi(z_{\eps_k}^i + \eps_k s)
        (z_{\eps_k}^i + \eps_k s)^{n-1} ds \nonumber\\[1ex]
        =\ & \phi(r_i) r_i^{n-1} \int_{-\infty}^{+\infty} \left( \dot{q}_0^2 + W(q_0) \right) ds 
        = \phi(r_i) r_i^{n-1} \sigma_W^{a,b},\label{eq:lim_mu_[c,d]}
    \end{align}
    which proves~(\ref{eq:conv-near-zeros-red}). Moreover, summing over $i\in I$, we obtain
    \[
    \sum_{i \in I} \phi(r_i) r_i^{n-1} \sigma_W^{a,b} \leq \norm*{\phi}_{L^{\infty}} \cdot \limsup_{k \to +\infty} E_{\eps_k}(u_{\eps_k},\R^n). \\[-1ex]
    \]
    Since this holds for every $\varphi \in C_0 (0,+\infty)$ we deduce (\ref{eq:somma_ri_finita}).
    
    \smallskip
    
    \textsc{Step 5:} We prove that for every $i < N_0$ it turns out that
    \begin{equation*}
        \mu_{\eps_k} \mres A_{d_k^{i+1}, c_k^i} \overset{*}{\rightharpoonup} 0.
    \end{equation*}
    Clearly, the conclusion follows if we prove that
    \begin{equation}
        \label{eq:step5-conv-norm}
        \lim_{k \to +\infty} \mu_{\eps_k} \left( A_{d_k^{i+1}, c_k^i} \right)=0. 
    \end{equation}
    
    Since
    $$\mu_{\eps_k} \left( A_{d_k^{i+1}, c_k^i} \right)= E_{\eps_k}        \left( u_{\eps_k}, A_{d_k^{i+1}, c_k^i}\right),$$
    equality (\ref{eq:step5-conv-norm}) is a direct consequence of Lemma~\ref{lemma:out-of-zeros} applied with $\Omega_{\eps_k} = A_{d_k^{i+1}, c_k^i}$ provided that we check the two assumptions. 
    
    Assumption (1) follows by Remark~\ref{remark:bdry-radial}, indeed~(\ref{eq:conv-bu}) implies
    \begin{align}
        \label{eq:limit-bdry-1}
        & \lim_{k \to +\infty} W(u_{\eps_k}(d_k^{i+1})) = \lim_{k \to +\infty} W(u_{\eps_k}(c_k^i)) = 0, \\[0.5ex]
        \label{eq:limit-bdry-2}
        & \lim_{k \to +\infty} \eps_k \dot{u}_{\eps_k}(d_k^{i+1}) = \lim_{k \to +\infty} \eps_k \dot{u}_{\eps_k}(c_k^i) = 0.
    \end{align}
    
    Let us suppose by contradiction that assumption (2) of Lemma~\ref{lemma:out-of-zeros} does not hold. Therefore, there exist $\delta_0 > 0$, a subsequence of $\{ u_{\eps_k} \}$ that we do not relabel, and a sequence of points $r_k \in A_{d_k^{i+1}, c_k^i}$ such that $a+\delta_0 < u_{\eps_k}(r_k) < b-\delta_0$.
    Moreover, in the interval $[d_k^{i+1},c_k^i]$ we know that either $u_{\eps_k} < \gamma_0$ or $u_{\eps_k} > \gamma_0$. Then, possibly extracting another subsequence, we can assume without loss of generality that $\gamma_0 < u_{\eps_k}(r_k) < b-\delta_0$. In this case we know that $u_{\eps_k}(d_k^{i+1}) \ge b-\delta_0$ and $u_{\eps}(c_k^i) \ge b-\delta_0$ for $k$ large enough, since they both tend to $b$ as $k \to +\infty$.
    In particular, if $\hat{r}_k$ is a minimizer of $u_{\eps_k}$ in the interval $[d_k^{i+1},c_k^i]$ then $\hat{r}_k$ is internal for $k$ large enough, indeed we have $u_{\eps_k}(\hat{r}_k) \le u_{\eps_k}(r_k) < b-\delta_0$ and $u_{\eps_k} \ge b-\delta_0$ on the boundary. 
    
    Up to a further subsequence we can assume that there exists $\gamma \in [\gamma_0, b-\delta_0]$ such that $u_{\eps_k}(\hat{r}_k) \to \gamma$ as $k \to +\infty$. Then by Proposition~\ref{prop:blowup} there exists a subsequence such that the blow-ups $\psi_{\eps_k}$ of $u_{\eps_k}$ at $\hat{r}_k$ converge in $W^{2,2}(-1,1)$ to either $q_0 \circ \tau_{\gamma}^+$ or $q_0 \circ \tau_{\gamma}^-$, but this is in contrast with $\dot{u}_{\eps_k}(\hat{r}_k)=0$. Hence we can apply Lemma~\ref{lemma:out-of-zeros} to deduce that~(\ref{eq:step5-conv-norm}) holds.
    
    \smallskip
    
    \textsc{Step 6:} We recall from \textsc{Step 1} that there exists $R > 0$ such that $d_k^1 \le R$ for every $k$. Now, we prove that
    \begin{equation*}
            \mu_{\eps_k} \mres \left(\mathbb{R}^n \setminus \overline{B}_{d_k^1}\right) \overset{*}{\rightharpoonup} 0.
    \end{equation*}
    Notice that
    \[
        \mathbb{R}^n \setminus \overline{B}_{d_k^1} = \bigcup_{L > R} B_L \setminus \overline{B}_{d_k^1} = \bigcup_{L > R} A_{L,d_k^1}.
    \]
    Therefore, as in \textsc{Step 5}, the conclusion follows if for every $L > R$ we prove that
    \begin{equation}
        \label{eq:step6-conv-norm}
        \lim_{k \to +\infty} \mu_{\eps_k} \Big(A_{L, d_k^1} \Big) = 0. 
    \end{equation}
    
    We observe that for every $k$ there exists $L_k \in [L,2L]$ such that
    \begin{align}
        \left( \eps_k \dot{u}_{\eps_k}(L_k)^2 + \frac{W(u_{\eps_k}(L_k))}{\eps_k} \right) L_k^{n-1}
        \notag
        & \le 
        \frac{1}{L}\int_L^{2L} \left( \eps_k \dot{u}_{\eps_k}(r)^2 + \frac{W(u_{\eps_k}(r))}{\eps_k} \right) r^{n-1} dr \\[1ex]
        \label{eq:step6-reductionL}
        & \le 
        \frac{E_{\eps_k}(u_{\eps_k}, \mathbb{R}^n)}{\omega_{n-1}L}.
    \end{align}
    In particular, using the elementary inequality $2\alpha \beta \le \alpha^2 + \beta^2$ we deduce that
    \begin{equation}
        \label{eq:step6-assumption1-1}
        \lim_{k \to +\infty} \eps_k \abs*{\dot{u}_{\eps_k}(L_k)}  \sqrt{W(u_{\eps_k}(L_k))} L_k^{n-1}
        \le \lim_{k \to +\infty} \frac{\eps_k}{2} \left( \eps_k \dot{u}_{\eps_k}(L_k)^2 + \frac{W(u_{\eps_k}(L_k))}{\eps_k} \right) L_k^{n-1} = 0,
    \end{equation}
    and also
    \begin{equation}
        \label{eq:step6-assumption1-2}
        \lim_{k \to +\infty} W(u_{\eps_k}(L_k)) \le \lim_{k \to +\infty} \eps_k \left( \eps_k \dot{u}_{\eps_k}(L_k)^2 + \frac{W(u_{\eps_k}(L_k))}{\eps_k} \right) = 0.
    \end{equation}
    
    It is clear from~(\ref{eq:step6-assumption1-2}) that the only limit points of $u_{\eps_k}(L_k)$ as $k \to +\infty$ are $a$ and $b$. On the other hand $L_k > R$ and on $[R,+\infty)$ we have either $u_{\eps_k} < \gamma_0$ or $u_{\eps_k} > \gamma_0$, therefore necessarily
    \begin{equation}
        \label{eq:step6-same-limit}
        \lim_{k \to +\infty} u_{\eps_k}(L_k) = \lim_{k \to +\infty} u_{\eps_k}(d_k^1) \in \{ a,b \}.
    \end{equation}
    Now we want to apply Lemma~\ref{lemma:out-of-zeros} with $\Omega_{\eps_k} = A_{L_k, d_k^1}$. As in \textsc{Step 5}, we have to check two assumptions. The first assumption follows combining Remark~\ref{remark:bdry-radial} and~(\ref{eq:step6-assumption1-1}) and also the analogues of~(\ref{eq:limit-bdry-1}) and~(\ref{eq:limit-bdry-2}) with $d_k^1$ in place of $d_k^{i+1}$ and $c_k^i$. The second assumption can be proved exactly as in \textsc{Step 5} because we know from~(\ref{eq:step6-same-limit}) that on the boundary of $[d_k^1, L_k]$ we have either $u_{\eps_k} \le a - \delta_0$ or $u_{\eps_k} \ge b - \delta_0$ for $k$ large enough.
    
    Therefore, Lemma~\ref{lemma:out-of-zeros} tells us that
    \[
        \lim_{k \to +\infty} E_{\eps_k} \Big (u_{\eps_k}, A_{L_k, d_k^1} \Big) = 0,
    \]
    which clearly implies~(\ref{eq:step6-conv-norm}).
    
    \smallskip
    
    \textsc{Step 7:} If $N_0 = N_{\bar{r}}$ for some $\bar{r} > 0$ (in particular $N_0 < +\infty$), then for every $r < \bar{r}$ we have
    \begin{equation*}
            \mu_{\eps_k} \mres A_{r,c_k^{N_0}} \overset{*}{\rightharpoonup} 0.
    \end{equation*}
    This is again a consequence of Lemma~\ref{lemma:out-of-zeros} and the proof is very similar to the one of \textsc{Step 6}. The starting point is to prove the analogous of~(\ref{eq:step6-reductionL}), that is for every $k$ there exists $r_k \in [r/2,r]$ such that
    \[
        \left( \eps_k \dot{u}_{\eps_k}(r_k)^2 + \frac{W(u_{\eps_k}(r_k))}{\eps_k} \right) r_k^{n-1} 
        \le
        \frac{ 2 E_{\eps_k}(u_{\eps_k},\mathbb{R}^n)}{\omega_{n-1}r}.
    \]
    Then we conclude as in the previous step using $r_k$ instead of $L_k$.
    
    Finally, combining \textsc{Steps 4-7} we deduce the validity of~(\ref{eq:int-var}).
    
    To prove (\ref{eq:discr_to_0}) it is enough to observe that $|\xi_\eps|\leq \mu_\eps$, and hence the convergences to zero in \textsc{Steps 5-7} also hold with $\xi_\eps$.
    
    Furthermore, if we repeat the computations in \textsc{Step 4} with $\xi_{\eps_k}$ instead of $\mu_{\eps_k}$, then (\ref{eq:lim_mu_[c,d]}) becomes
    $$\lim_{k \to +\infty} \int_{c_k^i}^{d_k^i} 
        \left( \eps_k \dot{u}_{\eps_k}^2  - \frac{W(u_{\eps_k})}{\eps_k} \right) \phi(r) r^{n-1} dr 
        = \phi(r_i) r_i^{n-1} \int_{-\infty}^{+\infty} \left( \dot{q}_0^2 - W(q_0) \right) ds =0.$$
        
    This proves that $\xi_{\eps_k}\overset{*}{\weakto} 0$ and, since the limit does not depend on the sequence $\{\eps_k\}$, we deduce that the whole family $\{\xi_\eps\}$ converges to zero.
    \end{proof}

\subsubsection*{\centering Acknowledgements}

The first and the second author are members of the ``Gruppo Nazionale per l'Analisi Matematica, la Probabilità e le loro Applicazioni'' (GNAMPA) of the ``Istituto Nazionale di Alta Matematica'' (INdAM).

The second and the third author wish to thank Massimo Gobbino, who introduced them to the study of the modified conjecture and related problems during their master thesis projects.

\printbibliography[heading=bibintoc]

\end{document}